\documentclass[11pt,reqno]{amsart}

\usepackage{amsmath, amsfonts, amsthm, amssymb}
\usepackage[usenames]{color}

\numberwithin{equation}{section}

\newtheorem{Theorem}{Theorem}[section]
\newtheorem{Lemma}{Lemma}[section]

\newtheorem{Proposition}{Proposition}[section]

\theoremstyle{definition}
\newtheorem{Definition}{Definition}[section]
\newtheorem{Remark}{Remark}[section]

\renewcommand{\r}{\rho}

\def\i{\varepsilon}
\renewcommand{\u}{{\bf u}}

\newcommand{\R}{{\mathbb R}}
\newcommand{\Dv}{{\rm div}}

\newcommand{\dl}{\delta}

\def\f{\frac}
\renewcommand{\O}{\omega}
\def\ov{\overline}

\def\D{\Delta }

\def\hf1{^\f{1}{1-\xi^2}}

\def\be{\begin{equation}}
\def\ee{\end{equation}}

\newcommand{\F}{{\mathtt F}}

\author{Xianpeng Hu and Dehua Wang}
\address{Department of Mathematics, University of Pittsburgh,
                           Pittsburgh, PA 15260, USA.}
\email{xih15@pitt.edu}
\address{Department of Mathematics, University of Pittsburgh,
                           Pittsburgh, PA 15260, USA.}
\email{dwang@math.pitt.edu}

\title[Compressible viscoelastic flows]
{Global Existence  for the Multi-Dimensional Compressible
Viscoelastic flows}

\date{April 30, 2010}

\begin{document}

\begin{abstract}
The global solutions in critical spaces to the multi-dimensional
compressible  viscoelastic flows are considered. The global
existence of the Cauchy problem with initial data close to an
equilibrium state is established in Besov spaces. Using uniform
estimates for a hyperbolic-parabolic linear system with convection
terms, we prove the global existence in the Besov space which is
invariant with respect to the {scaling} of the associated
equations. Several important estimates are achieved, including a
smoothing effect on the velocity, and the $L^1-$decay  of the
density and  deformation gradient.
\end{abstract}

\maketitle

\section{Introduction}
We consider the following equations of multi-dimensional
compressible viscoelastic flows \cite{CD, Gurtin, LZ, RHN}:
\begin{subequations} \label{e1e}
\begin{align}
&\widehat{\r}_t +\Dv(\widehat{\r}\widehat{\u})=0,\label{e1e1}\\
&(\widehat{\r}\widehat{\u})_t+\Dv\left(\widehat{\r}\widehat{\u}\otimes\widehat{\u}\right)-\mu\D \widehat{\u}-(\lambda+\mu)\nabla\Dv\widehat{\u}
+\nabla P(\widehat{\r})=\alpha\Dv(\widehat{\r}\, \F\,\F^\top),\label{e1e2}\\
&\F_t+\widehat{\u}\cdot\nabla\F=\nabla\widehat{\u} \,
\F,\label{e1e3}
\end{align}
\end{subequations}
where $\widehat{\r}$ stands for the density, $\widehat{\u}\in
\R^N$ $(N=2,3)$ the velocity, and $\F\in M^{N\times N}$ (the set
of $N\times N$ matrices)  the deformation gradient. The viscosity
coefficients $\mu, \lambda$ are two constants satisfying $\mu>0,
2\mu+N\lambda>0$, which ensures that the operator $-\mu\D
\widehat{\u}-(\lambda+\mu)\nabla\Dv\widehat{\u}$ is a strongly
elliptic operator. The pressure term $P(\widehat{\r})$ is an
increasing and convex function of $\widehat{\rho}$ for
$\widehat{\rho}>0$. The symbol $\otimes$ denotes the Kronecker
tensor product, $\F^\top$ means the transpose matrix of $\F$, and
the notation $\widehat{\u}\cdot\nabla\F$ is understood to be
$(\widehat{\u}\cdot\nabla)\F$.
For system \eqref{e1e}, the
corresponding elastic energy is chosen to be  the special form of
the Hookean linear elasticity:
$$W(\F)=\frac{\alpha}{2}|\F|^2,\quad \alpha>0,$$
which, however, does not reduce the essential difficulties for
analysis. The methods and results of this paper can be applied to
more general cases.

In this paper, we consider the Cauchy problem of system \eqref{e1e} subject to the initial condition:
\begin{equation}\label{IC1}
(\widehat{\r}, \widehat{\u}, \F)|_{t=0}=(\widehat{\r}_0(x),
\widehat{\u}_0(x), \F_0(x)), \quad x\in\R^N,
\end{equation}
and we are  interested  in the global existence and uniqueness of
strong solution to the initial-value problem
\eqref{e1e}-\eqref{IC1} near its equilibrium state in the
multi-dimensional space $\R^N$. Here the {equilibrium state} of
the system \eqref{e1e} is defined as: $\widehat{\r}$ is a positive
constant (for simplicity, $\widehat{\r}=1$), $\widehat{\u}=0$, and
$\F=I$ (the identity matrix in $M^{3\times 3}$). We introduce a
new unknown variable $E$ by setting
$$\F=I+E.$$ Then, \eqref{e1e} becomes
\begin{subequations} \label{e1}
\begin{align}
&\widehat{\r}_t +\Dv(\widehat{\r}\widehat{\u})=0,\label{e11}\\
&(\widehat{\r}\widehat{\u})_t+\Dv\left(\widehat{\r}\widehat{\u}\otimes\widehat{\u}\right)-\mu\D \widehat{\u}-(\mu+\lambda)\nabla\Dv\widehat{\u}
+\nabla P(\widehat{\r})=\alpha\Dv(\widehat{\r} (I+E)(I+E)^\top),\label{e12}\\
&E_t+\widehat{\u}\cdot\nabla E=\nabla\widehat{\u}
E+\nabla\widehat{\u},\label{e13}
\end{align}
\end{subequations}
with the initial data
\begin{equation}\label{IC}
(\widehat{\r}, \widehat{\u}, E)|_{t=0}=(\widehat{\r}_0(x),
\widehat{\u}_0(x), E_0(x)), \quad x\in\R^N.
\end{equation}

There have been some results about the local existence of strong
solutions to the compressible viscoelastic flows, see \cite{HW,
LZ} and the references therein. The global existence to
\eqref{e1e} is a difficult problem due to the appearance of the
deformation gradient. The challenge is to identify an appropriate
functional space where the Cauchy problem \eqref{e1e}-\eqref{IC1}
is well-posed globally in time. In this paper, to construct a
global solution, we are going to use the scaling for the
compressible  viscoelastic flow to guess which space may be
critical. We observe that system \eqref{e1e} is invariant under
the transformation
$$(\widehat{\r}_0(x), \widehat{\u}_0(x), \F_0(x))\rightarrow (\widehat{\r}_0(lx), l\widehat{\u}_0(lx),
\F_0(lx)),$$
$$(\widehat{\r}(t,x), \widehat{\u}(t,x), \F(t,x))\rightarrow (\widehat{\r}(l^2t, lx),
l\widehat{\u}(l^2t,lx), \F(l^2t, lx)),$$ up to changes of the
pressure law $P$ into $l^2P$, and $\alpha$ into $l^2\alpha$. This
suggests the following definition:
\textit{A functional space $\mathfrak{A}\subset
\mathcal{S}'(\R^N)\times(\mathcal{S}'(\R^N))^N\times(\mathcal{S}'(\R^N))^{N\times
N}$ is called a critical space if the associated norm is invariant
under the transformation $(\r, \u, \F)\rightarrow (\r(l\cdot),
l\u(l\cdot), \F(l\cdot))$ (up to a constant independent of $l$),
where $\mathcal{S}'$ is the space of tempered distributions, i.e.,
the dual of the Schwartz space $\mathcal{S}$.}
According to this definition,
$B^{\f{N}{2}}\times(B^{\f{N}{2}-1})^N\times B^{\f{N}{2}}$ (see
Section 2 for the definition of
$B^s:=\dot{B}^s_{2,1}(\R^N)$) is a critical space.
The motivations to use the homogeneous Besov space $B^s$ with the
derivative index $\f{N}{2}$ include two points: first,
$B^{\f{N}{2}}$ is an algebra embedded in $L^\infty$, which allows
us to control the density and the deformation gradient from below
and from above without requiring more regularity on derivatives of
$\widehat{\r}$ and $\F$; second, the product is continuous from
$B^{\f{N}{2}-\alpha}\times B^{\f{N}{2}}$ to $B^{\f{N}{2}-\alpha}$
for $0\le\alpha<N$.

For the global existence, the hardest part of the argument is to
deal with the linear terms $\nabla\widehat{\r}$, $\Dv E$ and
$\nabla\widehat{\u}$, especially the first two terms. It turns out
that finding some dissipation for $\Dv E$ is a crucial step. This
step for the incompressible case has been fulfilled successfully
in \cite{LLZ}. For the compressible system \eqref{e1e}, we will
reformulate the system, and use the divergence-free property of
compressible viscoelastic flows for the ``compressible'' part of
the velocity, while the property on curl is used to deal with the
``incompressible'' part of the velocity. Meanwhile, we decompose
the deformation gradient into two parts: the symmetric part and
the antisymmetric part. With this technique and decomposition, we
will be able to obtain successfully the dissipation estimates on
the density and the deformation gradient for an auxiliary system
with convection terms. These estimates are crucial for the global
existence. We remark that for the global existence of solutions to
\eqref{e1e} near equilibrium, the intrinsic properties of the
divergence and the curl are important and necessary. During the
final stage of this paper, we noticed that some similar results
are also obtained independently in \cite{QZ}, where
 the intrinsic properties of the divergence and curl of viscoelastic flows (see Appendix) to system \eqref{e1e} are used to control the dissipation of the deformation gradient $\F$.

For the incompressible viscoelastic flows and related models, there are many papers in literature on classical solutions (cf. \cite{CM, CZ, KP, LLZH2,
LZP} and the references therein). On the other hand, the global existence of weak solutions
to the incompressible viscoelastic flows with large initial data
is still an outstanding open question, although there are some
progress in that direction (\cite{LLZH, LM, LW}).
For the well-posedness of global solutions to the compressible
Navier-Stokes equations, see \cite{RD3, RD1} (for barotropic
cases), and \cite{RD2} (for barotropic cases with heat
conduction). For the inviscid elastodynamics, see \cite{ST} and
their references on the global existence of classical solutions.

The rest of this paper is organized as follows. In Section 2, we
review the definitions of Besov spaces and  show
some good property of the Besov spaces. In Section 3, we
reformulate the system \eqref{e1e} and state the main theorem.
Section 4 is devoted to \textrm{a priori} estimates for an
auxiliary linear system with convection terms. In Section 5, we give the
proof of our main result, while in the appendix (Section 6), we prove two
intrinsic properties of compressible viscoelastic flows.

\bigskip

\section{Basic Properties of Besov Spaces}

Throughout this paper,  we use $C$ for a generic constant, and
denote  $A\le CB$ by  $A\lesssim B$. The notation $A\thickapprox
B$ means that $A\lesssim B$ and $B\lesssim A$. Also we use
$(\alpha_q)_{q\in\mathbb{Z}}$ to denote a sequence such that
$\sum_{q\in\mathbb{Z}}\alpha_q\le 1$. $(f|g)$ denotes the inner
product of two functions $f, g$ in $L^2(\R^N)$. The standard
summation notation over the repeated index is adopted in this
paper.

The definition of homogeneous Besov spaces is built on an
homogeneous Littlewood-Paley decomposition. First, we introduce a
function $\psi\in C^\infty(\R^N)$, supported in the shell
$$\mathcal{C}=\{\xi\in\R^N: \f{5}{6}\le|\xi|\le\f{12}{5}\},$$
 such that
$$\sum_{q\in\mathbb{Z}}\psi(2^{-q}\xi)=1,\textrm{ if }\xi\neq 0.$$
Denoting  by $h:=\mathcal{F}^{-1}\psi$ the inverse Fourier transform  of $\psi$, we define the dyadic blocks as follows:
$$\D_q f=\psi(2^{-q}D)f=2^{qN}\int_{\R^N}h(2^qy)f(x-y)dy,$$
and
$$S_q f=\sum_{p\le q-1}\D_pf,$$
where $D$ is the first order differential operator. The formal
decomposition
\begin{equation}\label{21}
f=\sum_{q\in\mathbb{Z}}\D_qf
\end{equation}
is called the homogeneous Littlewood-Paley decomposition. But
unfortunately, the above identity is not always true in
$\mathcal{S}'(\R^N)$ as pointed out in \cite{RD1}. Nevertheless,
\eqref{21} is true modulo polynimials (see \cite{CH, RD, JP}).

For $s\in\R$ and $f\in \mathcal{S}'(\R^N)$, we denote
$$\|f\|_{B^s}:=\sum_{q\in\mathbb{Z}}2^{sq}\|\D_qf\|_{L^2}.$$
Notice that $\|\cdot\|_{B^s}$ is only a semi-norm on
$\{f\in\mathcal{S}'(\R^N):\|f\|_{B^s}<\infty\}$, because
$\|f\|_{B^s}$ vanishes if and only if $f$ is a polynomial. This
leads us to introduce the following definition for homogeneous
Besov spaces:
\begin{Definition}
Let $s\in\R$ and $m=-\left[\f{N}{2}+1-s\right]$. If $m<0$, we set
$$B^s=\left\{f\in \mathcal{S}'(\R^N):\|f\|_{B^s}<\infty\textrm{
and }f=\sum_{q\in\mathbb{Z}}\D_qf\textrm{ in
}\mathcal{S}'(\R^N)\right\}.$$ If $m\ge 0$, we denote by
$\mathcal{P}_m$ the set of  polynomials with $N$ variables of degree
$\le m$ and define
$$B^s=\left\{f\in \mathcal{S}'(\R^N)/\mathcal{P}_m:\|f\|_{B^s}<\infty\textrm{
and }f=\sum_{q\in\mathbb{Z}}\D_qf\textrm{ in
}\mathcal{S}'(\R^N)/\mathcal{P}_m\right\}.$$
\end{Definition}

Functions in $B^s$ have many good properties (see Proposition 2.5
in \cite{RD1}):
\begin{Proposition}\label{p3}
The following properties hold:
\begin{itemize}
\item Density: the set $C_0^\infty$ is dense in $B^s$ if $|s|\le
\f{N}{2}$;
\item Derivation: $\|f\|_{B^s}\thickapprox \|\nabla f\|_{B^{s-1}}$;
\item Fractional derivation: let $\Gamma=\sqrt{-\D}$ and
$\sigma\in\R$; then the operator $\Gamma^\sigma$ is an isomorphism
from $B^s$ to $B^{s-\sigma}$;
\item Algebraic properties: for $s>0$, $B^s\cap L^\infty$ is an
algebra;
\item Interpolation: $(B^{s_1}, B^{s_2})_{\theta,1}=B^{\theta
s_1+(1-\theta)s_2}$.
\end{itemize}
\end{Proposition}

For the composition in $B^s$, we refer to \cite{RD} for the proof
of the following estimates:
\begin{Lemma}\label{l11}
Given $s>0$ and $f\in L^\infty\cap B^s$.
\begin{itemize}
\item Let $\Psi\in W^{[s]+2}_{loc}(\R^N)$ such that $\Psi(0)=0$.
Then $\Psi(f)\in B^s$. Moreover, there exists a function $C$ of
one variable depending only on $s$, $N$ and $\Psi$, and such that
$$\|\Psi(f)\|_{B^s}\le C(\|f\|_{L^\infty})\|f\|_{B^s}.$$
\item Let $\Phi\in W^{[s]+2}_{loc}(\R^N)$ such that $\Phi'(0)=0$.
Suppose that $f$ and $g$ belong to $B^{\f{N}{2}}$ and that
$(f-g)\in B^s$ for some $s\in(-\f{N}{2},\f{N}{2}]$. Then
$\Phi(f)-\Phi(g)$ belongs to $B^s$ and there exists a function of
two variables C depending only on $s,N$ and $\Phi$, and such that
$$\|\Phi(f)-\Phi(g)\|_{B^s}\le C(\|f\|_{L^\infty},
\|g\|_{L^\infty})\left(\|f\|_{B^{\f{N}{2}}}+\|g\|_{B^{\f{N}{2}}}\right)\|f-g\|_{B^s}.$$
\end{itemize}
\end{Lemma}

But, different from the nonhomogeneous Besov space, the
homogeneous Besov spaces fail to have nice inclusion properties.
For example, owing to the low frequencies, the inclusion
$B^s\hookrightarrow B^r$ does not hold for $s>r$. Still, the
functions of $B^s$ are locally more regular than those of $B^r$:
for any $\varphi\in C_0^\infty$ and $f\in B^s$, the function
$\varphi f$ is in $B^r$. This motivates the definition of
\textit{hybrid Besov spaces} where the growth conditions satisfied
by the dyadic blocks are not the same for low and high
frequencies. Let us recall that using hybrid Besov spaces has been
crucial for proving global well-posedness for compressible gases
in critical spaces (see \cite{RD1, RD2}).
The definition of the hybird Besov space is given as follows (see
Definition 2.8 in \cite{RD1} or \cite{RD2}).
\begin{Definition}
Let $s,t\in \R$. We set
$$\|f\|_{\tilde{B}^{s,t}}=\sum_{q\le
0}2^{qs}\|\D_qf\|_{L^2}+\sum_{q>0}2^{qt}\|\D_qf\|_{L^2}.$$
Denoting $m=-\left[\f{N}{2}+1-s\right]$, we define
$$\tilde{B}^{s,t}=\left\{f\in\mathcal{S}'(\R^N):   \|f\|_{\tilde{B}^{s,t}}<\infty\right\}\textrm{
if }m<0,$$
$$\tilde{B}^{s,t}=\left\{f\in\mathcal{S}'(\R^N)/\mathcal{P}_m:   \|f\|_{\tilde{B}^{s,t}}<\infty\right\}\textrm{
if }m\ge 0,$$
\end{Definition}

\begin{Remark}
Some remarks about the hybrid Besov spaces are in order:
\begin{itemize}
\item $\tilde{B}^{s,s}=B^s$;
\item If $s\le t$, then $\tilde{B}^{s,t}=B^s\cap B^t$. Otherwise,
$\tilde{B}^{s,t}=B^s+B^t$. In particular,
$\tilde{B}^{s,\f{N}{2}}\hookrightarrow L^\infty$ as $s\le
\f{N}{2}$;
\item The space $\tilde{B}^{0,s}$ coincides with the usual
nonhomogeneous Besov space
$$\left\{f\in\mathcal{S}'(\R^N):   \|\chi(D)f\|_{L^2}+\sum_{q\ge
0}2^{qs}\|\D_qf\|_{L^2}<\infty\right\}, \;\text{where}\;
\chi(\xi)=1-\sum_{q\ge 0}\phi(2^{-q}\xi);$$
\item If $s_1\le s_2$ and $t_1\ge t_2$, then
$\tilde{B}^{s_1,t_1}\hookrightarrow\tilde{B}^{s_2,t_2}$.
\end{itemize}
\end{Remark}
For  products of functions in hybrid Besov spaces, we have (see Proposition
2.10 in \cite{RD1}):
\begin{Proposition}\label{p2}
Given $s_1, s_2, t_1, t_2 \in \R$.
\begin{itemize}
\item For all $s_1, s_2>0$,
$$\|fg\|_{\tilde{B}^{s_1,s_2}}\lesssim
\|f\|_{L^\infty}\|g\|_{\tilde{B}^{s_1,s_2}}+\|g\|_{L^\infty}\|f\|_{\tilde{B}^{s_1,s_2}}.$$
\item
For all $s_1, s_2\le\f{N}{2}$ such that $\min\{s_1+t_1,
s_2+t_2\}>0$,
$$\|fg\|_{\tilde{B}^{s_1+s_2-\f{N}{2},t_1+t_2-\f{N}{2}}}\lesssim
\|f\|_{\tilde{B}^{s_1,s_2}}\|g\|_{\tilde{B}^{t_1,t_2}}.$$
\end{itemize}
\end{Proposition}

In order to state our existence result, we introduce some
functional spaces and explain the notations. Let $T>0$,
$r\in[0,\infty]$ and $X$ be a Banach space. We denote by
$\mathcal{M}(0,T;X)$ the set of measurable functions on $(0,T)$
valued in $X$. For $f\in \mathcal{M}(0,T; X)$, we define
$$\|f\|_{L^r_T(X)}=\left(\int_0^T\|f(\tau)\|_X^rd\tau\right)^{\f{1}{r}}\textrm{
if }r<\infty,$$
$$\|f\|_{L^\infty_T(X)}=\sup \textrm{ess}_{\tau\in(0,T)}\|f(\tau)\|_X.$$
Denote
$$L^r(0,T;X)=\{f\in \mathcal{M}(0,T;X):   \|f\|_{L^r_T(X)}<\infty\}.$$
 If $T=\infty$, we
denote by $L^r(\R^+; X)$ and $\|f\|_{L^r(X)}$ the corresponding
spaces and norms. Also denote by $C([0,T],X)$ (or $C(\R^+,X)$) the
set of continuous X-valued functions on $[0,T]$ (resp. $\R^+$). We
shall further denote by $C_b(\R^+;X)$ the set of bounded
continuous X-valued functions.

For $\alpha\in(0,1)$, $C^\alpha([0,T];X)$ (or $C^\alpha(\R^+;X)$)
stands for the space of the H\"older continuous functions in time with
order $\alpha$, that is, for every $t,s$ in $[0,T]$ (resp.
$\R^+$), we have
$$\|f(t)-f(s)\|_X\lesssim |t-s|^\alpha.$$

In this paper, the following estimates for the convection
terms arising in the localized system will be used several times (cf.
Lemma 5.1 in \cite{RD2} or Lemma 6.2 in \cite{RD1}).

\begin{Lemma}\label{cl}
Let $G$ be an homogeneous smooth function of degree $m$. Suppose
$-\f{N}{2}<s_i,t_i\le 1+\f{N}{2}$ for $i=1,2$. Then the following
 inequalities hold:
 \begin{equation}\label{25}
\begin{split}
&|(G(D)\D_q(\u\cdot\nabla f)|G(D)\D_q f)| \\
&\quad \le C\alpha_q
2^{-q(\phi^{s_1,s_2}(q)-m)}
\|\u\|_{B^{1+\f{N}{2}}}\|f\|_{\tilde{B}^{s_1,s_2}}\left\|G(D)\D_q
f)\right\|_{L^2},
\end{split}
\end{equation}
and
 \begin{equation}\label{26}
\begin{split}
&\left|(G(D)\D_q(\u\cdot\nabla f)|\D_q g)+(\D_q(\u\cdot\nabla
g)|G(D)\D_q f)\right|\\
&\quad\le
C\alpha_q\|\u\|_{B^{1+\f{N}{2}}}\Big(2^{-q(\phi^{t_1,t_2}(q)-m)}\left\|G(D)\D_qf\right\|_{L^2}\|g\|_{\tilde{B}^{t_1,t_2}}\\
&\qquad+2^{-q(\phi^{s_1,s_2}-m)}
\|f\|_{\tilde{B}^{s_1,s_2}}\|\D_qg\|_{L^2}\Big),
\end{split}
\end{equation}
where 
\begin{equation*}
\phi^{s,t}(q):=
\begin{cases}
s,\quad\textrm{if}\quad q\le 0,\\
t,\quad\textrm{if}\quad q\ge 1.
\end{cases}
\end{equation*}
\end{Lemma}

For the nonlinear term $\nabla\u E$, we have the following
estimates.
\begin{Lemma}\label{l222}
If  $-\f{N}{2}<s_i\le 1+\f{N}{2}$ for $i=1,2$, then
\begin{equation}\label{27}
\begin{split}
\|\nabla\u E\|_{\tilde{B}^{s_1,s_2}}\le C
\|\u\|_{B^{1+\f{N}{2}}}\|E\|_{\tilde{B}^{s_1,s_2}}.
\end{split}
\end{equation}
\end{Lemma}

To proof the above lemma, we need to recall the paradifferential
calculus which enables us to define a generalized product between
distributions. The paraproduct between $f$ and $g$ is defined by
$$T_f g=\sum_{q\in\mathbb{Z}}S_{q-1}f\D_qg.$$
We also have the following formal decomposition (modulo a polynomial):
$$fg=T_f g+T_g f+R(f,g),$$
with
$$R(f,g)=\sum_{q\in\mathbb{Z}}\D_q f\tilde{\D}_q g,$$
where $\tilde{\D}_q=\D_{q-1}+\D_q+\D_{q+1}.$

\begin{proof}[Proof of Lemma \ref{l222}]
Denoting $T'_fg=T_fg+R(g, f)$, we get the following decomposition
\begin{equation}\label{29}
\D_q(\nabla \u E)=\D_qT'_E \nabla\u+J_q,
\end{equation}
with
\begin{equation*}
\begin{split}
J_q=\sum_{|q'-q|\le 3}\Big([\D_q, S_{q'-1}(\nabla
\u)]\D_{q'}E+(S_{q'-1}-S_{q-1})(\nabla\u)\D_q\D_{q'}E+S_{q-1}(\nabla\u)\D_q
E\Big).
\end{split}
\end{equation*}
Applying Proposition 5.2 in \cite{RD2} or Proposition 6.1 in
\cite{RD1}, we see that the first term on the right-hand side of
\eqref{29} satisfies \eqref{27} provided $-\f{N}{2}<s_i\le
\f{N}{2}+1$ for $i=1,2$. Next we estimate $J_q$ term by term.
First, for the commutator $[\D_q, S_{q'-1}(\nabla \u)]\D_{q'}E$,
we have
\begin{equation*}
\begin{split}
&[\D_q, S_{q'-1}(\nabla
\u)]\D_{q'}E(x)\\&\quad=2^{-q}\int_{\R^N}\int_0^1h(y)(y\cdot
S_{q'-1}(\nabla\nabla \u)(x-2^{-q}\tau y))\D_{q'}E(x-2^{-q}y)d\tau
dy.
\end{split}
\end{equation*}
The above identity, together with Young's inequality for
convolution operator, yields
$$\|[\D_q, S_{q'-1}(\nabla
\u)]\D_{q'}E\|_{L^2}\le
C\|\nabla\u\|_{L^\infty}\|\D_{q'}E\|_{L^2},$$ since
$$\|S_{q'-1}\nabla\nabla\u\|_{L^\infty}\lesssim
2^q\|\nabla\u\|_{L^\infty}\lesssim 2^q\|\u\|_{B^{\f{N}{2}+1}}$$
according to Bernstein's Lemma (cf. \cite{CH, RD}). Hence, we
easily obtain the following inequality:
$$\|J_q\|_{L^2}\le C\alpha_q 2^{-q\phi^{s_1,s_2}(q)}
\|\u\|_{B^{1+\f{N}{2}}}\|E\|_{\tilde{B}^{s_1,s_2}}.$$

The proof is complete.
\end{proof}

\bigskip

\section{Reformulation and Main Results}

In this section, we state our global existence result. We first
reformulate  system \eqref{e1e}. Assume that the pressure
$P(\widehat{\r})$ is an increasing convex function with
$P'(1)>0$, and denote $\chi_0=(P'(1))^{-\f{1}{2}}$. For
$\widehat{\r}>0$, system \eqref{e1e} can be rewritten as
\begin{subequations} \label{x61}
\begin{align}
&\widehat{\r}_t +\widehat{\u}\cdot\nabla\widehat{\r}+\widehat{\r}\Dv\widehat{\u}=0,\\
&\partial_t\widehat{\u}_i+\widehat{\u}\cdot\nabla\widehat{\u}_i-\f{1}{\widehat{\r}}\left(\mu\D
\widehat{\u}_i-(\lambda+\mu)\partial_{x_i}\Dv\widehat{\u}\right)
+\f{P'(\widehat{\r})}{\widehat{\r}}\partial_{x_i} \widehat{\r}=\alpha\F_{jk}\partial_{x_j}\F_{ik},\\
&\F_t+\widehat{\u}\cdot\nabla\F=\nabla\widehat{\u} \, \F,
\end{align}
\end{subequations}
where we used the condition $\Dv(\widehat{\r} \F^\top)=0$ (see
Lemma \ref{div}) for all $t\ge 0$, which ensures that the i-th
component of the vector $\Dv(\r\F\F^\top)$ is
\begin{equation*}
\begin{split}
\partial_{x_j}(\widehat{\r}\F_{ik}\F_{jk})&=\widehat{\r}\F_{jk}\partial_{x_j}\F_{ik}+\F_{ik}\partial_{x_j}(\widehat{\r}\F_{jk})\\
&=\widehat{\r}\F_{jk}\partial_{x_j}\F_{ik}.
\end{split}
\end{equation*}

Define
$$\r(t,x)=\widehat{\r}(\chi_0^2t,\chi_0 x)-1,\quad \u(t,x)=\chi_0
\widehat{\u}(\chi_0^2 t, \chi_0 x),\quad E(t,x)=\F(\chi_0^2
t,\chi_0 x)-I,$$
then
\begin{subequations} \label{62000}
\begin{align}
&\r_t +\u\cdot\nabla\r+\Dv\u=-\r\Dv\u,\\
&\partial_t\u_i+\u\cdot\nabla\u_i-\mathcal{A}\u
+\nabla_{x_i}\r-a\partial_{x_j} E_{ij}=aE_{jk}\partial_{x_j}E_{ik}-\f{\r}{1+\r}\mathcal{A}\u-K(\r)\partial_{x_i}\r,\\
&E_t+\u\cdot\nabla E-\nabla\u=\nabla\u E,
\end{align}
\end{subequations}
with
$$K(\r):= \f{P'(\r+1)}{(1+\r)P'(1)}-1,\quad
\mathcal{A}:=\mu\D+(\lambda+\mu)\nabla\Dv,\quad
a=\f{\alpha}{P'(1)}.$$ We remark that for simplicity of the
presentation, we will assume that $a=1$ for the rest of this
paper.

For $s\in\R$, we denote
$$\Lambda^sf:=\mathcal{F}^{-1}(|\xi|^s\mathcal{F}(f)).$$
Let
$$d=\Lambda^{-1}\Dv\u$$ be the``compressible part" of the velocity,
and $$\O=\Lambda^{-1}\textrm{curl}\u, \quad\text{with }
\textrm{curl}(\u))_i^j=\partial_{x_j}\u^i-\partial_{x_i}\u^j$$
be the ``incompressible part". Setting $\nu=\lambda+2\mu$,  then system
\eqref{62000} can be rewritten as
\begin{subequations} \label{63}
\begin{align}
&\r_t +\Lambda d=-\r\Dv\u-\u\cdot\nabla\r,\\
&\partial_t d-\nu\D d-\Lambda \r-\mathcal{T} E=\Lambda^{-1}\Dv\left(-\u\cdot\nabla \u+E_{jk}\partial_{x_j}E_{ik}-\f{\r}{1+\r}\mathcal{A}\u-K(\r)\partial_{x_i}\r\right),\\
&\partial_t\O-\mu\D\O-\mathcal{R}E=\Lambda^{-1}\textrm{curl}\left(-\u\cdot\nabla \u+E_{jk}\partial_{x_j}E_{ik}-\f{\r}{1+\r}\mathcal{A}\u-K(\r)\partial_{x_i}\r\right),\\
&E_t+\u\cdot\nabla E-\nabla\u=\nabla\u E,\\
&\u=-\Lambda^{-1}\nabla d+\Lambda^{-1}\textrm{curl}\O,
\end{align}
\end{subequations}
where
$$\mathcal{R}=\Lambda^{-1}\textrm{curl}\,\Dv, \quad
\mathcal{T}=\Lambda^{-1}\Dv\,\Dv.$$
The operators $\Lambda$, $\mathcal{T}$ and
$\mathcal{R}$ are differential operators of order one.

Notice that the condition $\Dv(\widehat{\r}\F)=0$ for all $t\ge 0$
implies that $\f{\partial^2(\widehat{\r}\F_{ij})}{\partial
x_i\partial x_j}=0$ for all $t\ge 0$ and smooth functions
$\widehat{\r}$, $\F$. Hence, we have
\begin{equation}\label{6401}
\begin{split}
\mathcal{T}E&=\Lambda^{-1}\left(\f{\partial^2E_{ij}}{\partial x_i\partial x_j}\right)\\
&=\underbrace{\Lambda^{-1}\left(\f{\partial^2[(1+\r)(\dl_{ij}+E_{ij})]}{\partial
x_i\partial x_j}\right)}_{=0}-\Lambda^{-1}\Dv\Dv(\r I+\r E)\\
&=\Lambda\r-\Lambda^{-1}\Dv\Dv(\r E),
\end{split}
\end{equation}
where
\begin{equation*} \dl_{ij}=\begin{cases}0,\quad
\textrm{if}\quad i\neq j;\\ 1,\quad\textrm{if}\quad
i=j.\end{cases}
\end{equation*}
On the other hand, according to Lemma \ref{curl} (see Appendix),
we have
\begin{equation*}
\begin{split}
(\mathcal{R}E)_{ij}&=\Lambda^{-1}\left(\partial_{x_j}\left(\f{\partial
E_{ik}}{\partial x_k}\right)-\partial_{x_i}\left(\f{\partial
E_{jk}}{\partial x_k}\right)\right)\\
&=\Lambda^{-1}\left(\partial_{x_k}\left(\f{\partial
E_{ik}}{\partial x_j}\right)-\partial_{x_k}\left(\f{\partial
E_{jk}}{\partial x_i}\right)\right)\\
&=\Lambda^{-1}\left(\partial_{x_k}\left(\f{\partial
E_{ij}}{\partial x_k}\right)-\partial_{x_k}\left(\f{\partial
E_{ji}}{\partial
x_k}\right)\right)+\Lambda^{-1}\partial_{x_k}(E_{lk}\nabla_lE_{ij}-E_{lj}\nabla_l
E_{ik})\\
&\quad-\Lambda^{-1}\partial_{x_k}(E_{lk}\nabla_lE_{ji}-E_{li}\nabla_l E_{jk})\\
&=-\Lambda (E_{ij}-E_{ji})
+\Lambda^{-1}\partial_{x_k}(E_{lk}\nabla_lE_{ij}-E_{lj}\nabla_l
E_{ik})\\&\quad-\Lambda^{-1}\partial_{x_k}(E_{lk}\nabla_lE_{ji}-E_{li}\nabla_l
E_{jk}).
\end{split}
\end{equation*}
Thus, we finally obtain
\begin{subequations} \label{x64}
\begin{align}
&\r_t +\Lambda d=-\r\Dv\u-\u\cdot\nabla\r,\\
&\partial_t d-\nu\D d-2\Lambda \r=\Lambda^{-1}\Dv\left(-\u\cdot\nabla \u+E_{jk}\partial_{x_j}E_{ik}
-\f{\r}{1+\r}\mathcal{A}\u-K(\r)\partial_{x_i}\r-\Dv(\r E)\right),\\
&\partial_t\O-\mu\D\O+\Lambda
(E-E^\top)=\Lambda^{-1}\textrm{curl}\left(-\u\cdot\nabla \u+E_{jk}
\partial_{x_j}E_{ik}-\f{\r}{1+\r}\mathcal{A}\u-K(\r)\partial_{x_i}\r\right)+\mathcal{S},\\
&(E^\top-E)_t+\u\cdot\nabla (E^\top-E)+\Lambda\O=(\nabla\u E)^\top-\nabla\u E,\\
&\mathcal{E}_t+2\Lambda d=-\Lambda^{-1}\partial_{x_i}\Lambda^{-1}\partial_{x_j}\left(\u\cdot\nabla(E_{ij}+E_{ji})\right)+\Lambda^{-1}\partial_{x_i}\Lambda^{-1}\partial_{x_j}\left((\nabla\u E)_{ij}+(\nabla\u E)_{ji}\right),\\
&\u=-\Lambda^{-1}\nabla
d+\Lambda^{-1}\textrm{curl}\O,
\end{align}
\end{subequations}
where the antisymmetric matrix $\mathcal{S}$ is defined as
$$\mathcal{S}_{ij}=\Lambda^{-1}\partial_{x_k}(E_{lk}\nabla_lE_{ij}-E_{lj}\nabla_l
E_{ik})-\Lambda^{-1}\partial_{x_k}(E_{lk}\nabla_lE_{ji}-E_{li}\nabla_l E_{jk}),$$
and the scalar function $\mathcal{E}$ is defined as
$$\mathcal{E}_{ij}=\Lambda^{-1}\partial_{x_i}\Lambda^{-1}\partial_{x_j}(E_{ij}+E_{ji}).$$
Notice that from Proposition \ref{p3}, we deduce that
$$\|\mathcal{E}\|_{B^s}\thickapprox \|E+E^\top\|_{B^s},$$ and
\begin{equation}\label{EE}
\|\mathcal{E}\|_{B^s}+\|E-E^\top\|_{B^s}\thickapprox
\|E\|_{B^s}.
\end{equation}
Also, according to \eqref{6401} and the second equation of
\eqref{x64}, we have
\begin{equation}\label{6402}
\partial_t d-\nu\D d-2\Lambda \mathcal{E}=\Lambda^{-1}\Dv\left(-\u\cdot\nabla \u+E_{jk}\partial_{x_j}E_{ik}
-\f{\r}{1+\r}\mathcal{A}\u-K(\r)\partial_{x_i}\r+\Dv(\r E)\right).
\end{equation}
The motivation to write the second equation of \eqref{x64} as
\eqref{6402} is to obtain the estimate on the symmetric part
$\mathcal{E}$ of the deformation gradient, as we will see in
section 4.

The fact that this new formulation (\eqref{x64}, supplemented with
\eqref{6402}) is equivalent to \eqref{x61} requires some
explanation: it is not immediately obvious that the second and the
third equation in the above system are equivalent to the second
equation in \eqref{x61} under the condition that
$\Dv((1+\r)(I+E)^\top)=0$. Notice however that the right hand side
(denote it by $\mathcal{O}$) of the third equation in \eqref{x64}
and the term $\Lambda(E^\top-E)$ are skew-symmetric matrices and
satisfy the following Jacobi relation:
$$\partial_{x_i}\mathcal{O}_j^k+\partial_{x_j}\mathcal{O}_k^i+\partial_{x_k}\mathcal{O}_i^j=0\quad\textrm{for}\quad
1\le i,j,k\le N.$$ Since
$\O_0:=\Lambda^{-1}\textrm{curl}\u_0$ is a skew-symmetric
matrix satisfying the Jacobi relation, this is also the case for
$\O$. We therefore have the equivalence
$$\u=-\Lambda^{-1}\nabla
d+\Lambda^{-1}\textrm{curl}\O\quad\Leftrightarrow\quad\Dv\u=\Lambda
d\quad\textrm{and}\quad\textrm{curl}\u=\Lambda\O,$$ which enables
us to conclude that $\u$ indeed satisfies the second equation
\eqref{x61} as soon as $d$ and $\O$ satisfy the second and the
third equation in \eqref{x64}.

The existence of a solution to \eqref{e1e} is proved thanks to a
classical (and tedious) iteration method: we define a sequence of
approximate solutions of \eqref{e1e} which solve a linear systems
to which Proposition \ref{p1} applies. For small enough initial
data, we obtain uniform estimates so that we can use a compactness
argument to show the convergence of such an approximate solution.
Refer to Section 5 for more details of the complete proof.

Let us now introduce the functional space which appears in the
global existence theorem.

\begin{Definition}
For $T>0$, and $s\in \R$, we denote
\begin{equation*}
\begin{split}
\mathfrak{B}^s_T&=\Big\{(\r, \u, E)\in \left(L^1(0,T;
\tilde{B}^{s+1,s})\cap
C([0,T];\tilde{B}^{s-1,s})\right)\\&\qquad\qquad\qquad\times\left(L^1(0,T;
B^{s+1})\cap
C([0,T];B^{s-1})\right)^N\\&\qquad\qquad\qquad\times\left(L^1(0,T;
\tilde{B}^{s+1,s})\cap C([0,T];\tilde{B}^{s-1,s})\right)^{N\times
N}\Big\}
\end{split}
\end{equation*}
and
\begin{equation*}
\begin{split}
\|(\r,\u,
E)\|_{\mathfrak{B}^s_T}&=\|\r\|_{L^\infty_T(\tilde{B}^{s-1,s})}+\|\u\|_{L^\infty_T(B^{s-1,s})}+\|E\|_{L^\infty_T(\tilde{B}^{s-1,s})}\\
&\quad+\|\r\|_{L^1_T(\tilde{B}^{s+1,s})}+\|\u\|_{L^1_T(B^{s+1})}+\|E\|_{L^1_T(\tilde{B}^{s+1,s})}.
\end{split}
\end{equation*}
We use the notation $\mathfrak{B}^s$ if $T=+\infty$ by changing
the interval $[0,T]$ into $[0,\infty)$ in the definition above.
\end{Definition}

Now it is ready to state our main result:

\begin{Theorem}\label{mt}
There exists two positive constants $\gamma$ and $\Gamma$,  such
that,  if $\widehat{\r}_0-1\in \tilde{B}^{\f{N}{2}-1,\f{N}{2}}$,
$\widehat{\u}_0\in B^{\f{N}{2}-1}$,
$\F_0-I\in\tilde{B}^{\f{N}{2}-1,\f{N}{2}}$ satisfy
\begin{itemize}
\item $\|\widehat{\r}_0-1\|_{\tilde{B}^{\f{N}{2}-1,\f{N}{2}}}+\|\widehat{\u}_0\|_{B^{\f{N}{2}-1}}+\|\F_0-I\|_{\tilde{B}^{\f{N}{2}-1,\f{N}{2}}}\le
\gamma;$
\item $\Dv(\widehat{\r}_0\F_0^\top)=0$;
\item
$\F_{lk}(0)\nabla_{x_l}\F_{ij}(0)=\F_{lj}(0)\nabla_{x_l}\F_{ik}(0),$
\end{itemize}
then system \eqref{e1e} has a solution $(\widehat{\r}, \widehat{\u}, \F)$ with $(\widehat{\r}-1,
\widehat{\u}, \F-I)$ in $\mathfrak{B}^{\f{N}{2}}$ satisfying
$$\|(\widehat{\r}-1,\widehat{\u},\F-I)\|_{\mathfrak{B}^{\f{N}{2}}}\le
\Gamma\left(\|\widehat{\r}_0-1\|_{\tilde{B}^{\f{N}{2}-1,\f{N}{2}}}+\|\widehat{\u}_0\|_{B^{\f{N}{2}-1}}+\|\F_0-I\|_{\tilde{B}^{\f{N}{2}-1,\f{N}{2}}}\right).$$
\end{Theorem}

\begin{Remark}
The solution in Theorem \ref{mt} is also unique, but we omit the
proof of the uniqueness.  The proof of uniqueness will be
same as in \cite{RD3} with a slightly modification due to the
deformation gradient. See also  \cite{QZ} for a proof.
\end{Remark}

\bigskip

\section{Estimates of an Linear Problem}

In this section, we consider the following auxiliary linear system:
\begin{equation}\label{31}
\begin{cases}
\partial_t \r+\u\cdot\nabla \r+\Lambda d=\mathfrak{L},\\
\partial_t d+\u\cdot\nabla d-\nu\D d-2\Lambda \r=\mathfrak{M},\\
\partial_t\O+\u\cdot\nabla\O-\mu\D\O+\Lambda(E-E^\top)=\mathfrak{N},\\
\partial_t (E^\top-E)+\u\cdot\nabla (E^\top-E)+\Lambda\O=\mathfrak{Q},\\
\mathcal{E}_t+\u\cdot\nabla\mathcal{E}+2\Lambda d=\mathfrak{K},
\end{cases}
\end{equation}
where $\mathfrak{L}, \mathfrak{M}, \mathfrak{N}, \mathfrak{Q},
\mathfrak{J}, \mathfrak{K}$, and $\u$ are given functions, and
\eqref{EE} gives a relation of $E, E-E^\top,$ and $\mathcal{E}$.
We remark that, as in \eqref{6402}, to obtain estimates on
$\mathcal{E}$, we need to rewrite the second equation of
\eqref{31} as
\begin{equation}\label{x31}
\partial_t d+\u\cdot\nabla d-\nu\D d-2\Lambda \mathcal{E}=\mathfrak{J},
\end{equation}
under the constraint
\begin{equation}\label{x32}
2\Lambda \r+\mathfrak{M}=2\Lambda \mathcal{E}+\mathfrak{J}.
\end{equation}

For this system, we have the following estimate:

\begin{Proposition}\label{p1}
Let $(\r, d, \O, E-E^\top, \mathcal{E})$ be a solution of \eqref{31} on $[0,T)$, and
$$V(t):=\int_0^t\|\u(s)\|_{B^{\f{N}{2}+1}}ds.$$
Under the condition \eqref{x32}, the following estimate holds  on
$[0,T)$:
\begin{equation*}
\begin{split}
&\|\r(t)\|_{\tilde{B}^{\f{N}{2}-1,\f{N}{2}}}+\|E(t)\|_{\tilde{B}^{\f{N}{2}-1,\f{N}{2}}}
+\|d(t)\|_{B^{\f{N}{2}-1}}+\|\O(t)\|_{B^{\f{N}{2}-1}}\\&\qquad+\int_0^t\left(\|\r(s)\|_{\tilde{B}^{\f{N}{2}+1,\f{N}{2}}}
+\|d(s)\|_{B^{\f{N}{2}+1}}+\|E(s)\|_{\tilde{B}^{\f{N}{2}+1,\f{N}{2}}}+\|\O(s)\|_{B^{\f{N}{2}+1}}\right)ds\\&\quad\le
Ce^{CV(t)}\Big\{\|\r_0\||_{\tilde{B}^{\f{N}{2}-1,\f{N}{2}}}+\|E_0\|_{\tilde{B}^{\f{N}{2}-1,\f{N}{2}}}
+\|d_0\|_{B^{\f{N}{2}-1}}+\|\O_0\|_{B^{\f{N}{2}-1}}\\
&\qquad+\int_0^te^{-CV(s)}\Big(\|\mathfrak{L}\|_{\tilde{B}^{\f{N}{2}-1,\f{N}{2}}}+\|\mathfrak{M}\|_{B^{\f{N}{2}-1}}+\|\mathfrak{N}\|_{B^{\f{N}{2}-1}}
+\|\mathfrak{Q}\|_{\tilde{B}^{\f{N}{2}-1,\f{N}{2}}}\\
&\quad\qquad+\|\mathfrak{J}\|_{B^{\f{N}{2}-1}}
+\|\mathfrak{K}\|_{\tilde{B}^{\f{N}{2}-1,\f{N}{2}}}\Big)ds\Big\},
\end{split}
\end{equation*}
where $C$ depends only on $N$.
\end{Proposition}

\begin{Remark}\label{X1}
Notice that the constraint \eqref{x32} is always satisfied by our
system \eqref{x64} in view of the divergence property of the
deformation gradient $E$, and $\mathfrak{J}$ will be given as the
right-hand side of \eqref{6402}. This implies that the estimates
in Proposition \ref{p1} also hold for solutions to \eqref{x64}.
\end{Remark}


\begin{proof}[Proof of Proposition \ref{p1}]
To prove this proposition, we first localize \eqref{31} in low and
high frequencies according to the {Littlewood-Paley}
decomposition. We then use an energy method to estimate each
dyadic block. To this end, we will divide our proof into four
steps.

Let $(\r, d, \O, E)$ be a solution of \eqref{31} and $K>0$. Define
$$\tilde{\r}=e^{-KV(t)}\r,\quad
\tilde{d}=e^{-KV(t)}d,\quad\tilde{E}^\top-\tilde{E}=e^{-KV(t)}(E^\top-E),$$$$\tilde{\O}=e^{-KV(t)}\O,
\quad\tilde{\mathcal{E}}=e^{-KV(t)}\mathcal{E},
$$ and
$$\tilde{\mathfrak{L}}=e^{-KV(t)}\mathfrak{L},\quad
\tilde{\mathfrak{M}}=e^{-KV(t)}\mathfrak{M},\quad\tilde{\mathfrak{N}}=e^{-KV(t)}\mathfrak{N},$$
$$\tilde{\mathfrak{J}}=e^{-KV(t)}\mathfrak{J},\quad
\tilde{\mathfrak{K}}=e^{-KV(t)}\mathfrak{K}.$$ Applying the
operator $\D_q$ to \eqref{31} and \eqref{x31}, we deduce that
$(\D_q\tilde{\r}, \D_q\tilde{d}, \D_q\tilde{\O}, \D_q\tilde{E})$
satisfies
\begin{equation}\label{32}
\begin{cases}
\partial_t \D_q\tilde{\r}+\D_q(\u\cdot\nabla \tilde{\r})+\Lambda\D_q \tilde{d}=\D_q\tilde{\mathfrak{L}}-KV'(t)\D_q\tilde{\r},\\
\partial_t \D_q\tilde{d}+\D_q(\u\cdot\nabla \tilde{d})-\nu\D \D_q\tilde{d}-2\Lambda\D_q \tilde{\r}
=\D_q\tilde{\mathfrak{M}}-KV'(t)\D_q\tilde{d},\\
\partial_t\D_q\tilde{\O}+\D_q(\u\cdot\nabla\tilde{\O})-\mu\D\D_q\tilde{\O}+\Lambda\D_q(\tilde{E}-\tilde{E}^\top)=\D_q\tilde{\mathfrak{N}}
-KV'(t)\D_q\tilde{\O},\\
\partial_t \D_q(\tilde{E}^\top-\tilde{E})+\D_q(\u\cdot\nabla
(\tilde{E}^\top-\tilde{E}))+\Lambda\tilde{\O}=\D_q\mathfrak{\tilde{Q}}
-KV'(t)\D_q(\tilde{E}^\top-\tilde{E}),\\
\partial_t \D_q\tilde{d}+\D_q(\u\cdot\nabla\tilde{d})-\nu\D \D_q\tilde{d}-2\Lambda\D_q \tilde{\mathcal{E}}=\D_q\tilde{\mathfrak{J}},\\
\partial_t\D_q\tilde{\mathcal{E}}+\D_q(\u\cdot\nabla\tilde{\mathcal{E}})+2\Lambda \D_q\tilde{d}=\D_q\tilde{\mathfrak{K}}.
\end{cases}
\end{equation}

Denote
\begin{equation*}
\begin{split}
g_q&:=2^{q\left(\f{N}{2}-1\right)}\Big(2\|\D_q\tilde{\r}\|^2_{L^2}+2\|\D_q\tilde{d}\|^2_{L^2}+\|\D_q(\tilde{E}^\top-\tilde{E})\|^2_{L^2}
+\|\D_q\tilde{\mathcal{E}}\|^2_{L^2}
+\|\D_q\tilde{\O}\|^2_{L^2}\\&\quad-\f{\nu}{\eta}(\Lambda\D_q(\tilde{E}^\top-\tilde{E})|\D_q\tilde{\O})
-\f{\nu}{\eta}(\Lambda\D_q\tilde{\r}|\D_q\tilde{d})-\f{\nu}{\eta}(\Lambda\D_q\tilde{\mathcal{E}}|\D_q\tilde{d}))\Big)^{\f{1}{2}}
\end{split}
\end{equation*}
for $q\le q_0$ with
$\eta=\max\left\{\f{4^{q_0}\nu^2+3}{2},\nu,\f{\nu}{\mu},\f{4^{q_0}\mu\nu}{2}\right\}+1$;
\begin{equation*}
\begin{split}
g_q&:=2^{q(\f{N}{2}-1)}\Big(\|\Lambda\D_q\tilde{\r}\|_{L^2}^2+\|\Lambda\D_q(\tilde{E}^\top-\tilde{E})\|_{L^2}^2
+\|\Lambda\D_q\tilde{\mathcal{E}}\|_{L^2}^2
+\|\D_q\tilde{d}\|_{L^2}^2+\|\D_q\tilde{\O}\|_{L^2}^2
\\&\qquad-(\Lambda\D_q\tilde{\r}|\D_q\tilde{d})-(\Lambda\D_q(\tilde{E}^\top-\tilde{E})|\D_q\tilde{\O})
-(\Lambda\D_q\tilde{\mathcal{E}}|\D_q\tilde{d})\Big)^{\f{1}{2}}
\end{split}
\end{equation*}
for $q>q_0$,  where $\beta_1=\f{2}{\nu}$, $\beta_2=\f{2}{\mu}$,
$\gamma=\max\left\{\f{2}{\mu^2},\f{5}{\nu^2}\right\}+1$,  and $q_0$
is chosen to satisfy
\begin{equation}\label{301}
\|\Lambda\D_qf\|_{L^2}\ge 2\gamma\|\D_q f\|_{L^2}\quad\textrm{for
all}\quad q\ge q_0.
\end{equation}
Due to the fact supp$\mathcal{F}(\D_q\r)\subset 2^q\mathcal{C}$
and supp$\mathcal{F}(\D_qE)\subset 2^q\mathcal{C}$, one deduce
that
\begin{equation}\label{33}
\left(\f{g_q}{2^{q\phi^{\f{N}{2}-1,\f{N}{2}}(q)}(\|\D_q\tilde{\r}\|_{L^2}+\|\D_q\tilde{E}\|_{L^2})
+2^{q(\f{N}{2}-1)}(\|\D_q\tilde{d}\|_{L^2}+\|\D_q\tilde{\O}\|_{L^2})}\right)^{\pm
1}\le C
\end{equation}
for a universal constant $C$.

The first two steps of the proof are devoted to getting the
following inequality:
\begin{equation}\label{34}
\begin{split}
&\f{1}{2}\f{d}{dt}g_q^2+\kappa 2^{q\phi^{N, N-2}(q)}
\Big(\|\D_q\tilde{\r}\|_{L^2}^2+\|\Lambda\D_q(\tilde{E}^\top-\tilde{E})\|_{L^2}^2+\|\Lambda\D_q\tilde{\mathcal{E}}\|_{L^2}^2
+\|\D_q\tilde{\O}\|_{L^2}^2+\|\D_q\tilde{d}\|_{L^2}^2\Big)\\
&\le C\alpha_q g_q\Big(\|\tilde{\mathfrak{L}}\|_{\tilde{B}^{\f{N}{2}-1,\f{N}{2}}}+\|\tilde{\mathfrak{Q}}\|_{\tilde{B}^{\f{N}{2}-1,\f{N}{2}}}
+\|\tilde{\mathfrak{K}}\|_{\tilde{B}^{\f{N}{2}-1,\f{N}{2}}}+\|\tilde{\mathfrak{J}}\|_{B^{\f{N}{2}-1}}+\|\tilde{\mathfrak{M}}\|_{B^{\f{N}{2}-1}}
+\|\tilde{\mathfrak{N}}\|_{B^{\f{N}{2}-1}}\\&\qquad+V'(t)(\|\tilde{\r}\|_{\tilde{B}^{\f{N}{2}-1,\f{N}{2}}}
+\|\tilde{E}^\top-\tilde{E}\|_{\tilde{B}^{\f{N}{2}-1,\f{N}{2}}}+\|\tilde{\mathcal{E}}\|_{\tilde{B}^{\f{N}{2}-1,\f{N}{2}}}+\|\tilde{d}\|_{B^{\f{N}{2}-1}}+\|\tilde{\O}\|_{B^{\f{N}{2}-1}})\Big)\\&\qquad-KV'g_q^2,
\end{split}
\end{equation}
where $\kappa$ is a universal constant.
\medskip

{\bf First Step: Low Frequencies.}\quad
Suppose $q\le q_0$ and define
\begin{equation*}
\begin{split}
f_q^2&=2\|\D_q\tilde{\r}\|^2_{L^2}+2\|\D_q\tilde{d}\|^2_{L^2}+\|\D_q(\tilde{E}^\top-\tilde{E})\|^2_{L^2}+\|\D_q\tilde{\mathcal{E}}\|^2_{L^2}
+\|\D_q\tilde{\O}\|^2_{L^2}\\&\quad-\f{\nu}{\eta}(\Lambda\D_q(\tilde{E}^\top-\tilde{E})|\D_q\tilde{\O})
-\f{\nu}{\eta}(\Lambda\D_q\tilde{\r}|\D_q\tilde{d})-\f{\nu}{\eta}(\Lambda\D_q\tilde{\mathcal{E}}|\D_q\tilde{d}).
\end{split}
\end{equation*}
Taking the $L^2$-scalar product of the first equation of \eqref{32}
with $\D_q\tilde{\r}$, of the second equation with
$\D_q\tilde{d}$, of the third equation with $\D_q\tilde{\O}$, and
of the fourth equation with $\D_q(\tilde{E}^\top-\tilde{E})$, we
obtain the following four identities:
\begin{equation}\label{35}
\begin{split}
\f{1}{2}\f{d}{dt}\|\D_q\tilde{\r}\|_{L^2}^2+(\D_q(\u\cdot\nabla\tilde{\r})|\D_q\tilde{\r})&+(\Lambda\D_q\tilde{d}|\D_q\tilde{\r})
=(\D_q\tilde{\mathfrak{L}}|\D_q\tilde{\r})-KV'\|\D_q\tilde{\r}\|_{L^2}^2;
\end{split}
\end{equation}
\begin{equation}\label{36}
\begin{split}
&\f{1}{2}\f{d}{dt}\|\D_q\tilde{d}\|_{L^2}^2+\nu\|\Lambda\D_q\tilde{d}\|_{L^2}^2+(\D_q(\u\cdot\nabla\tilde{d})|\D_q\tilde{d})
-2(\Lambda\D_q\tilde{\r}|\D_q\tilde{d})\\&\quad=(\D_q\tilde{\mathfrak{M}}|\D_q\tilde{d})-KV'\|\D_q\tilde{d}\|_{L^2}^2;
\end{split}
\end{equation}
\begin{equation}\label{3601}
\begin{split}
&\f{1}{2}\f{d}{dt}\|\D_q\tilde{\O}\|_{L^2}+(\D_q(\u\cdot\nabla\tilde{\O})|\D_q\tilde{\O})+\mu\|\Lambda\D_q\tilde{\O}\|^2_{L^2}
+(\Lambda\D_q(\tilde{E}-\tilde{E}^\top)|\D_q\tilde{\O})\\
&\quad=(\D_q\tilde{\mathfrak{N}}|\D_q\tilde{\O})-KV'\|\D_q\tilde{\O}\|_{L^2}^2;
\end{split}
\end{equation}
\begin{equation}\label{37}
\begin{split}
&\f{1}{2}\f{d}{dt}\|\D_q(\tilde{E}^\top-\tilde{E})\|_{L^2}^2+(\D_q((\u\cdot\nabla(\tilde{E}^\top-\tilde{E})))|\D_q(\tilde{E}^\top-\tilde{E}))
+(\Lambda\D_q\tilde{\O}|\D_q(\tilde{E}^\top-\tilde{E}))
\\&\quad=(\D_q(\mathfrak{\tilde{Q}}^\top-\mathfrak{\tilde{Q}})|\D_q(\tilde{E}^\top-\tilde{E}))-KV'\|(\tilde{E}^\top-\tilde{E})\|_{L^2}^2.
\end{split}
\end{equation}
And, we also have, from fifth and sixth equation in \eqref{32}
\begin{equation}\label{3701}
\begin{split}
&\f{1}{2}\f{d}{dt}\|\D_q\tilde{d}\|_{L^2}^2+\nu\|\Lambda\D_q\tilde{d}\|_{L^2}^2+(\D_q(\u\cdot\nabla\tilde{d})|\D_q\tilde{d})
-2(\Lambda\D_q\tilde{\mathcal{E}}|\D_q\tilde{d})\\&\quad=(\D_q\tilde{\mathfrak{J}}|\D_q\tilde{d})-KV'\|\D_q\tilde{d}\|_{L^2}^2;
\end{split}
\end{equation}
\begin{equation}\label{3702}
\begin{split}
&\f{1}{2}\f{d}{dt}\|\D_q\tilde{\mathcal{E}}\|_{L^2}^2+(\D_q(\u\cdot\nabla\tilde{\mathcal{E}})|\D_q\tilde{\mathcal{E}})
+2(\Lambda\D_q\tilde{d}|\D_q\tilde{\mathcal{E}})
=(\D_q\mathfrak{\tilde{K}}|\D_q\tilde{\mathcal{E}})-KV'\|\tilde{\mathcal{E}}\|_{L^2}^2.
\end{split}
\end{equation}

For estimates of the term $(\Lambda\D_q\tilde{\r}|\D_q\tilde{d})$,
we apply $\Lambda$ to the first equation in \eqref{32} and take
the $L^2$-scalar product with $\D_q\tilde{d}$, then take the
scalar product of the second equation with $\Lambda\D_q\tilde{\r}$
and sum both equalities to yield
\begin{equation}\label{38}
\begin{split}
&\f{d}{dt}(\Lambda\D_q\tilde{\r}|\D_q\tilde{d})+\|\Lambda\D_q\tilde{d}\|_{L^2}^2
-2\|\Lambda\D_q\tilde{\r}\|_{L^2}^2+(\Lambda\D_q\tilde{\r}|\D_q(\u\cdot\nabla\tilde{d}))
\\&\qquad+(\Lambda\D_q(\u\cdot\nabla\tilde{\r})|\D_q\tilde{d})+\nu(\Lambda^2\D_q\tilde{d}|\Lambda\D_q\tilde{\r})
\\
&\quad=(\Lambda\D_q\tilde{\mathfrak{L}}|\D_q\tilde{d})+(\Lambda\D_q\tilde{\r}|\D_q\tilde{\mathfrak{M}})
-2KV'(\Lambda\D_q\tilde{\r}|\D_q\tilde{d}).
\end{split}
\end{equation}

For estimates of the term
$(\Lambda\D_q(\tilde{E}^\top-\tilde{E})|\D_q\tilde{\O})$, we apply
$\Lambda$ to the fourth equation in \eqref{32} and take the
$L^2$-scalar product with $\D_q\tilde{\O}$, then take the scalar
product of the third equation with
$\Lambda\D_q(\tilde{E}^\top-\tilde{E})$ and sum both equalities to
yield
\begin{equation}\label{3901}
\begin{split}
&\f{d}{dt}(\Lambda\D_q(\tilde{E}^\top-\tilde{E})|\D_q\tilde{\O})-\|\Lambda\D_q(\tilde{E}^\top-\tilde{E})\|_{L^2}^2
+\|\Lambda\D_q\tilde{\O}\|_{L^2}^2+(\Lambda\D_q(\tilde{E}^\top-\tilde{E})|\D_q(\u\cdot\nabla\tilde{\O}))
\\&\qquad+(\Lambda\D_q(\u\cdot\nabla(\tilde{E}^\top-\tilde{E}))|\D_q\tilde{\O})
+\mu(\Lambda^2\D_q\tilde{\O}|\Lambda\D_q(\tilde{E}^\top-\tilde{E}))\\
&\quad=(\Lambda\D_q\mathfrak{\tilde{Q}}|\D_q\tilde{\O})
+(\Lambda\D_q(\tilde{E}^\top-\tilde{E})|\D_q\tilde{\mathfrak{N}})-2KV'(\Lambda\D_q(\tilde{E}^\top-\tilde{E})|\D_q\tilde{\O}).
\end{split}
\end{equation}

For estimates of the term
$(\Lambda\D_q\tilde{\mathcal{E}}|\D_q\tilde{d})$, we apply
$\Lambda$ to the last equation in \eqref{32} and take the
$L^2$-scalar product with $\D_q\tilde{d}$, then take the scalar
product of the fifth equation with
$\Lambda\D_q\tilde{\mathcal{E}}$ and sum both equalities to yield
\begin{equation}\label{3902}
\begin{split}
&\f{d}{dt}(\Lambda\D_q\tilde{\mathcal{E}}|\D_q\tilde{d})+\|\Lambda\D_q\tilde{d}\|_{L^2}^2
-2\|\Lambda\D_q\tilde{\mathcal{E}}\|_{L^2}^2+(\Lambda\D_q\tilde{\mathcal{E}}|\D_q(\u\cdot\nabla\tilde{d}))
\\&\qquad+(\Lambda\D_q(\u\cdot\nabla\tilde{\mathcal{E}})|\D_q\tilde{d})+\nu(\Lambda^2\D_q\tilde{d}|\Lambda\D_q\tilde{\mathcal{E}})
\\
&\quad=(\Lambda\D_q\tilde{\mathfrak{K}}|\D_q\tilde{d})+(\Lambda\D_q\tilde{\mathcal{E}}|\D_q\tilde{\mathfrak{J}})
-2KV'(\Lambda\D_q\tilde{\mathcal{E}}|\D_q\tilde{d}).
\end{split}
\end{equation}

Taking linear combination of \eqref{35}-\eqref{3902}, we obtain,
\begin{equation}\label{3101}
\begin{split}
&\f{1}{2}\f{d}{dt}f_q^2+\f{\nu}{2\eta}\Big(2\|\Lambda\D_q\tilde{\r}\|_{L^2}^2
+\|\Lambda\D_q(\tilde{E}^\top-\tilde{E})\|_{L^2}^2+2\|\Lambda\D_q\tilde{\mathcal{E}}\|_{L^2}^2+(2\eta-3)\|\Lambda\D_q\tilde{d}\|_{L^2}
\\&\qquad-(\nu\Lambda^2\D_q\tilde{d}|\Lambda\D_q\tilde{\r})-(\mu\Lambda^2\D_q\tilde{\O}|\Lambda\D_q(\tilde{E}^\top-\tilde{E})
-\nu(\Lambda^2\D_q\tilde{d}|\Lambda\D_q\tilde{\mathcal{E}}))\Big)+KV'f_q^2
\\&\qquad+\left(\mu-\f{\nu}{2\eta}\right)\|\Lambda\D_q\tilde{\O}\|_{L^2}^2\\
&\quad=\mathcal{X},
\end{split}
\end{equation}
where
\begin{equation*}
\begin{split}
\mathcal{X}&:=2(\D_q\tilde{\mathfrak{L}}|\D_q\tilde{\r})
+(\D_q\tilde{\mathfrak{M}}|\D_q\tilde{d})+(\D_q\tilde{\mathfrak{N}}|\D_q\tilde{\O})+(\D_q\tilde{\mathfrak{Q}}|\D_q(\tilde{E}^\top-\tilde{E}))
+(\D_q\tilde{\mathfrak{J}}|\D_q\tilde{d})\\&\qquad+(\D_q\tilde{\mathfrak{K}}|\D_q\tilde{\mathcal{E}})-2(\D_q(\u\cdot\nabla\tilde{\r})|\D_q\tilde{\r})
-2(\D_q(\u\cdot\nabla\tilde{d})|\D_q\tilde{d})
\\&\qquad-(\D_q(\u\cdot\nabla(\tilde{E}^\top-\tilde{E}))|\D_q(\tilde{E}^\top-\tilde{E}))-(\D_q(\u\cdot\nabla\tilde{\O})|\D_q\tilde{\O})
-(\D_q(\u\cdot\nabla\tilde{\mathcal{E}})|\D_q\tilde{\mathcal{E}})
\\&\qquad+\f{\nu}{2\eta}\Big\{(\Lambda\D_q\tilde{\r}|\D_q(\u\cdot\nabla\tilde{d}))
+(\Lambda\D_q(\u\cdot\nabla\tilde{\r})|\D_q\tilde{d})
+(\Lambda\D_q(\tilde{E}^\top-\tilde{E})|\D_q(\u\cdot\nabla\tilde{\O}))
\\&\qquad+(\Lambda\D_q(\u\cdot\nabla(\tilde{E}^\top-\tilde{E}))|\D_q\tilde{\O})-(\Lambda\D_q\tilde{\mathfrak{L}}|\D_q\tilde{d})
-(\Lambda\D_q(\tilde{E}^\top-\tilde{E})|\D_q\tilde{\mathfrak{N}})\\&\qquad-(\Lambda\D_q\tilde{\r}|\D_q\tilde{\mathfrak{M}})
-(\Lambda\D_q\mathfrak{\tilde{Q}}|\D_q\tilde{\O})+(\Lambda\D_q\tilde{\mathcal{E}}|\D_q(\u\cdot\nabla\tilde{d}))
+(\Lambda\D_q(\u\cdot\nabla\tilde{\mathcal{E}})|\D_q\tilde{d})\\&\qquad-(\Lambda\D_q\tilde{K}|\D_q\tilde{d})
-(\Lambda\D_q\tilde{\mathcal{E}}|\D_q\tilde{J})\Big\}.
\end{split}
\end{equation*}

As $q\le q_0$, there exists a constant $c_0\ge 1$, which depends
on $\lambda,\mu$, such that
\begin{equation}\label{311}
\f{1}{c_0}f_q^2\le\|\D_q\tilde{\r}\|_{L^2}^2+\|\D_q\tilde{d}\|_{L^2}^2+\|\D_q(\tilde{E}^\top-\tilde{E})\|_{L^2}^2+\|\D_q\tilde{\mathcal{E}}\|_{L^2}^2
+\|\D_q\tilde{\O}\|_{L^2}^2\le
c_0f_q^2,
\end{equation}
and this equivalence implies that there exists a universal
positive constant $\kappa$ depending on $\lambda$ and $\mu$, such
that
\begin{equation}\label{312}
\begin{split}
&\f{\nu}{2\eta}\Big(2\|\Lambda\D_q\tilde{\r}\|_{L^2}^2
+\|\Lambda\D_q(\tilde{E}^\top-\tilde{E})\|_{L^2}^2+2\|\Lambda\D_q\tilde{\mathcal{E}}\|_{L^2}^2+(2\eta-3)\|\Lambda\D_q\tilde{d}\|_{L^2}\\
&\qquad-(\nu\Lambda^2\D_q\tilde{d}|\Lambda\D_q\tilde{\r})-(\mu\Lambda^2\D_q\tilde{\O}|\Lambda\D_q(\tilde{E}^\top-\tilde{E})
-\nu(\Lambda^2\D_q\tilde{d}|\Lambda\D_q\tilde{\mathcal{E}}))\Big)\\
&\quad+\left(\mu-\f{\nu}{2\eta}\right)\|\Lambda\D_q\tilde{\O}\|_{L^2}^2
\\&\quad\ge\kappa
2^{2q}\Big(\|\D_q\tilde{\r}\|_{L^2}^2+\|\D_q(\tilde{E}^\top-\tilde{E})\|_{L^2}^2+\|\D_q\tilde{\mathcal{E}}\|_{L^2}^2
+\|\D_q\tilde{\O}\|_{L^2}^2+\|\D_q\tilde{d}\|_{L^2}^2\Big).
\end{split}
\end{equation}

For terms on the right-hand side of \eqref{3101}, we use Lemma
\ref{cl}, \eqref{311}, and the Cauchy-Schwarz inequality to obtain
\begin{equation}\label{313}
\begin{split}
\left|\mathcal{X}\right|&\le C
f_q\Big(\|\D_q\tilde{\mathfrak{L}}\|_{L^2}+\|\D_q\tilde{\mathfrak{M}}\|_{L^2}+\|\D_q\tilde{\mathfrak{N}}\|_{L^2}+\|\D_q\tilde{\mathfrak{Q}}\|_{L^2}
+\|\D_q\tilde{\mathfrak{K}}\|_{L^2}+\|\D_q\tilde{\mathfrak{J}}\|_{L^2}
\\&\qquad\qquad+2^{-q(\f{N}{2}-1)}\alpha_q
V'(\|\tilde{\r}\|_{\tilde{B}^{\f{N}{2}-1,\f{N}{2}}}+\|\tilde{E}^\top-\tilde{E}\|_{\tilde{B}^{\f{N}{2}-1,\f{N}{2}}}
+\|\tilde{\mathcal{E}}\|_{\tilde{B}^{\f{N}{2}-1,\f{N}{2}}}
\\&\qquad\qquad+\|\tilde{d}\|_{B^{\f{N}{2}-1}}+\|\tilde{\O}\|_{B^{\f{N}{2}-1}})\Big).
\end{split}
\end{equation}
Hence, combining \eqref{3101}, \eqref{311}, \eqref{312},
\eqref{313} together, we obtain
\begin{equation}\label{315}
\begin{split}
&\f{1}{2}\f{d}{dt}f_q^2+\kappa
2^{2q}\Big(\|\D_q\tilde{\r}\|_{L^2}^2+\|\Lambda\D_q(\tilde{E}^\top-\tilde{E})\|_{L^2}^2+\|\Lambda\D_q\tilde{\mathcal{E}}\|_{L^2}^2
+\|\D_q\tilde{\O}\|_{L^2}^2+\|\D_q\tilde{d}\|_{L^2}^2\Big)\\
&\le Cf_q\Big(\|\D_q\tilde{\mathfrak{L}}\|_{L^2}+\|\D_q\tilde{\mathfrak{M}}\|_{L^2}+\|\D_q\tilde{\mathfrak{N}}\|_{L^2}+\|\D_q\tilde{\mathfrak{Q}}\|_{L^2}
+\|\D_q\tilde{\mathfrak{K}}\|_{L^2}+\|\D_q\tilde{\mathfrak{J}}\|_{L^2}\\&\qquad+2^{-q(\f{N}{2}-1)}\alpha_q
V'(\|\tilde{\r}\|_{\tilde{B}^{\f{N}{2}-1,\f{N}{2}}}+\|\tilde{E}^\top-\tilde{E}\|_{\tilde{B}^{\f{N}{2}-1,\f{N}{2}}}
+\|\tilde{\mathcal{E}}\|_{\tilde{B}^{\f{N}{2}-1,\f{N}{2}}}
+\|\tilde{d}\|_{B^{\f{N}{2}-1}}\\&\qquad+\|\tilde{\O}\|_{B^{\f{N}{2}-1}}
)\Big)-KV'f_q^2.
\end{split}
\end{equation}

{\bf Second Step: High Frequencies.}\quad
In this step, we assume $q>q_0$. We apply the operator
$\Lambda$ to the first equation of \eqref{32}, multiply by
$\Lambda\D_q\tilde{\r}$ and integrate over $\R^N$ to yield
\begin{equation}\label{316}
\begin{split}
\f{1}{2}\f{d}{dt}\|\Lambda\D_q\tilde{\r}\|_{L^2}^2&+(\Lambda\D_q(\u\cdot\nabla\tilde{\r})|\Lambda\D_q{\tilde{\r}})
+(\Lambda^2\D_q\tilde{d}|\Lambda\D_q\tilde{\r})
\\&=(\Lambda\D_q\tilde{\mathfrak{L}}|\Lambda\D_q\tilde{\r})-KV'\|\Lambda\D_q\tilde{\r}\|^2_{L^2}.
\end{split}
\end{equation}
Applying the operator $\Lambda$ to the fourth equation of
\eqref{32}, multiplying by $\Lambda\D_q(\tilde{E}^\top-\tilde{E})$
and integrating over $\R^N$, we get
\begin{equation}\label{317}
\begin{split}
\f{1}{2}\f{d}{dt}\|\Lambda\D_q(\tilde{E}^\top-\tilde{E})\|_{L^2}^2&+(\Lambda\D_q(\u\cdot\nabla(\tilde{E}^\top-\tilde{E}))|\Lambda\D_q(\tilde{E}^\top-\tilde{E}))
+(\Lambda^2\D_q\tilde{\O}|\Lambda\D_q(\tilde{E}^\top-\tilde{E}))
\\&=(\Lambda\D_q\mathfrak{\tilde{Q}}|\Lambda\D_q(\tilde{E}^\top-\tilde{E}))-KV'\|\Lambda\D_q(\tilde{E}^\top-\tilde{E})\|^2_{L^2}.
\end{split}
\end{equation}
Applying the operator $\Lambda$ to the sixth equation of
\eqref{32}, multiplying by $\Lambda\D_q\tilde{\mathcal{E}}$ and
integrating over $\R^N$, we obtain
\begin{equation}\label{31701}
\begin{split}
\f{1}{2}\f{d}{dt}\|\Lambda\D_q\mathcal{\tilde{E}}\|_{L^2}^2&+(\Lambda\D_q(\u\cdot\nabla\tilde{\mathcal{E}})|\Lambda\D_q{\mathcal{\tilde{E}}})
+2(\Lambda^2\D_q\tilde{d}|\Lambda\D_q\mathcal{\tilde{E}})
\\&=(\Lambda\D_q\mathfrak{\tilde{K}}|\Lambda\D_q\mathcal{\tilde{E}})-KV'\|\Lambda\D_q\mathcal{\tilde{E}}\|^2_{L^2}.
\end{split}
\end{equation}
Denoting
\begin{equation*}
\begin{split}
f_q^2&=\|\Lambda\D_q\tilde{\r}\|_{L^2}^2+\|\Lambda\D_q(\tilde{E}^\top-\tilde{E})\|_{L^2}^2+\|\Lambda\D_q\tilde{\mathcal{E}}\|_{L^2}^2
+2\gamma\|\D_q\tilde{d}\|_{L^2}^2+\gamma\|\D_q\tilde{\O}\|_{L^2}^2
\\&\quad-\beta_1(\Lambda\D_q\tilde{\r}|\D_q\tilde{d})-\beta_2(\Lambda\D_q(\tilde{E}^\top-\tilde{E})|\D_q\tilde{\O})
-2\beta_1(\Lambda\D_q\tilde{\mathcal{E}}|\D_q\tilde{d}).
\end{split}
\end{equation*}
Combining \eqref{36}, \eqref{3601}, \eqref{3701}, \eqref{38},
\eqref{3901}, \eqref{316}, and \eqref{31701}, we obtain
\begin{equation}\label{3181}
\begin{split}
&\f{1}{2}\f{d}{dt}f_q^2+\beta_1\|\Lambda\D_q\tilde{\r}\|^2_{L^2}+2\beta_1\|\Lambda\D_q\tilde{\mathcal{E}}\|^2_{L^2}+\left(2\gamma\nu-\f{5\beta_1}{2}\right)\|\Lambda\D_q\tilde{d}\|_{L^2}^2
+\left(\mu\gamma-\f{\beta_2}{2}\right)\|\Lambda\D_q\tilde{\O}\|_{L^2}^2\\&\qquad+\f{\beta_2}{2}\|\Lambda\D_q(\tilde{E}^\top-\tilde{E})\|^2_{L^2}
-2\gamma(\Lambda\D_q\tilde{\r}|\D_q\tilde{d})-2\gamma(\Lambda\D_q\tilde{\mathcal{E}}|\D_q\tilde{d})
-\gamma(\Lambda\D_q(\tilde{E}^\top-\tilde{E})|\D_q\tilde{\O})\\&\qquad+KV'f_q^2\\&\quad=\mathcal{Y},
\end{split}
\end{equation}
where
\begin{equation*}
\begin{split}
\mathcal{Y}&=\gamma(\D_q\tilde{\mathfrak{M}}|\D_q\tilde{d})+\gamma(\D_q\tilde{\mathfrak{J}}|\D_q\tilde{d})+\gamma(\D_q\tilde{\mathfrak{N}}|\D_q\tilde{\O})-2\gamma(\D_q(\u\cdot\nabla\tilde{d})|\D_q\tilde{d})
\\&\quad-\gamma(\D_q(\u\cdot\nabla\tilde{\O})|\D_q\tilde{\O})+\f{\beta_1}{2}\Big((\Lambda\D_q\tilde{\r}|\D_q(\u\cdot\nabla\tilde{d}))
+(\Lambda\D_q(\u\cdot\nabla\tilde{\r})|\D_q\tilde{d})-(\Lambda\D_q\tilde{\mathfrak{L}}|\D_q\tilde{d})
\\&\quad-(\Lambda\D_q\tilde{\r}|\D_q\tilde{\mathfrak{M}})+2(\Lambda\D_q\tilde{\mathcal{E}}|\D_q(\u\cdot\nabla\tilde{d}))
+2(\Lambda\D_q(\u\cdot\nabla\tilde{\mathcal{E}})|\D_q\tilde{d})-2(\Lambda\D_q\tilde{\mathfrak{K}}|\D_q\tilde{d})\\&\quad-2(\Lambda\D_q\tilde{\mathcal{E}}|\D_q\tilde{\mathfrak{J}}))\Big)
+\f{\beta_2}{2}\Big((\Lambda\D_q(\tilde{E}^\top-\tilde{E})|\D_q(\u\cdot\nabla\tilde{\O}))
+(\Lambda\D_q(\u\cdot\nabla(\tilde{E}^\top-\tilde{E}))|\D_q\tilde{\O})\\
&\quad-(\Lambda\D_q\mathfrak{\tilde{Q}}|\D_q\tilde{\O})
-(\Lambda\D_q(\tilde{E}^\top-\tilde{E})|\D_q\tilde{\mathfrak{N}})\Big)-(\Lambda\D_q(\u\cdot\nabla\tilde{\r})|\Lambda\D_q{\tilde{\r}})
\\&\quad+(\Lambda\D_q\tilde{\mathfrak{L}}|\Lambda\D_q\tilde{\r})-(\Lambda\D_q(\u\cdot\nabla(\tilde{E}^\top-\tilde{E}))|\Lambda\D_q(\tilde{E}^\top-\tilde{E}))
+(\Lambda\D_q\mathfrak{\tilde{Q}}|\Lambda\D_q(\tilde{E}^\top-\tilde{E}))\\&\quad-(\Lambda\D_q(\u\cdot\nabla\tilde{\mathcal{E}})|\Lambda\D_q{\tilde{\mathcal{E}}})
+(\Lambda\D_q\tilde{\mathfrak{K}}|\Lambda\D_q\tilde{\mathcal{E}}).
\end{split}
\end{equation*}
Notice that, for $q\ge q_0$, we have
\begin{equation}\label{31812}
\begin{split}
f_q^2\thickapprox
\|\Lambda\D_q\tilde{\r}\|_{L^2}^2+\|\Lambda\D_q\tilde{E}\|_{L^2}^2+\|\D_q\tilde{d}\|_{L^2}^2+\|\D_q\tilde{\O}\|_{L^2}^2;
\end{split}
\end{equation}
and
\begin{equation}\label{31813}
\begin{split}
&\beta_1\|\Lambda\D_q\tilde{\r}\|^2_{L^2}+2\beta_1\|\Lambda\D_q\tilde{\mathcal{E}}\|^2_{L^2}+\left(2\gamma\nu-\f{5\beta_1}{2}\right)\|\Lambda\D_q\tilde{d}\|_{L^2}^2
+\left(\mu\gamma-\f{\beta_2}{2}\right)\|\Lambda\D_q\tilde{\O}\|_{L^2}^2\\&\quad+\f{\beta_2}{2}\|\Lambda\D_q(\tilde{E}^\top-\tilde{E})\|^2_{L^2}
-2\gamma(\Lambda\D_q\tilde{\r}|\D_q\tilde{d})-2\gamma(\Lambda\D_q\tilde{\mathcal{E}}|\D_q\tilde{d})
-\gamma(\Lambda\D_q(\tilde{E}^\top-\tilde{E})|\D_q\tilde{\O})\\
&\thickapprox\|\Lambda\D_q\tilde{\r}\|^2_{L^2}+\|\D_q\tilde{d}\|_{L^2}^2+\|\D_q\tilde{\O}\|_{L^2}^2
+\|\Lambda\D_q(\tilde{E}^\top-\tilde{E})\|^2_{L^2}+\|\Lambda\D_q\tilde{\mathcal{E}}\|^2_{L^2}.
\end{split}
\end{equation}
Next, we apply Lemma \ref{cl} to obtain, using \eqref{31812}
\begin{equation}\label{319}
\begin{split}
|\mathcal{Y}|&\le
Cf_q\Big(\|\Lambda\D_q\tilde{\mathfrak{L}}\|_{L^2}+\|\Lambda\D_q\tilde{\mathfrak{Q}}\|_{L^2}+\|\Lambda\D_q\tilde{\mathfrak{K}}\|_{L^2}+\|\D_q\tilde{\mathfrak{J}}\|_{L^2}+\|\D_q\tilde{\mathfrak{M}}\|_{L^2}
+\|\D_q\tilde{\mathfrak{N}}\|_{L^2}\\&\qquad\qquad+\alpha_q
2^{-q(\f{N}{2}-1)}V'(\|\tilde{\r}\|_{\tilde{B}^{\f{N}{2}-1,\f{N}{2}}}+\|\tilde{d}\|_{B^{\f{N}{2}-1}}+\|\tilde{E}^\top-\tilde{E}\|_{\tilde{B}^{\f{N}{2}-1,\f{N}{2}}}
\\&\qquad\qquad+\|\tilde{\mathcal{E}}\|_{\tilde{B}^{\f{N}{2}-1,\f{N}{2}}}+\|\tilde{\O}\|_{B^{\f{N}{2}-1}})\Big).
\end{split}
\end{equation}
Therefore, there exists a universal positive constant $\kappa$,
which depends on $\mu$ and $\nu$, such that,
\begin{equation}\label{320}
\begin{split}
&\f{1}{2}\f{d}{dt}f_q^2+\kappa
\left(\|\Lambda\D_q\tilde{\r}\|_{L^2}^2+\|\Lambda\D_q\tilde{E}\|_{L^2}^2+\|\D_q\tilde{d}\|_{L^2}^2+\|\D_q\tilde{\O}\|_{L^2}^2\right)\\
&\quad\le
Cf_q\Big(\|\Lambda\D_q\tilde{\mathfrak{L}}\|_{L^2}+\|\Lambda\D_q\tilde{\mathfrak{Q}}\|_{L^2}+\|\Lambda\D_q\tilde{\mathfrak{K}}\|_{L^2}+\|\D_q\tilde{\mathfrak{J}}\|_{L^2}+\|\D_q\tilde{\mathfrak{M}}\|_{L^2}
+\|\D_q\tilde{\mathfrak{N}}\|_{L^2}\\&\qquad\qquad+\alpha_q
2^{-q(\f{N}{2}-1)}V'(\|\tilde{\r}\|_{\tilde{B}^{\f{N}{2}-1,\f{N}{2}}}+\|\tilde{d}\|_{B^{\f{N}{2}-1}}+\|\tilde{E}^\top-\tilde{E}\|_{\tilde{B}^{\f{N}{2}-1,\f{N}{2}}}
\\&\qquad\qquad+\|\tilde{\mathcal{E}}\|_{\tilde{B}^{\f{N}{2}-1,\f{N}{2}}}+\|\tilde{\O}\|_{B^{\f{N}{2}-1}})\Big)-KV'f_q^2.
\end{split}
\end{equation}

{\bf Third Step: Damping Effect.}\quad
We are now going to show that inequality \eqref{34} entails a
decay for $\r$, $E$, $d$ and $\O$. We postpone the proof of
smoothing properties for $d$ and $\O$ to the next step. Let
$\dl>0$ be a small parameter (which will tend to 0) and denote
$h_q^2=g_q^2+\dl^2$. From \eqref{34}, and dividing by $h_q$, we
obtain
\begin{equation*}
\begin{split}
&\f{d}{dt}h_q+\kappa h_q\\
&\le
C\alpha_q\Big(\|\tilde{\mathfrak{L}}\|_{\tilde{B}^{\f{N}{2}-1,\f{N}{2}}}+\|\tilde{\mathfrak{Q}}\|_{\tilde{B}^{\f{N}{2}-1,\f{N}{2}}}
+\|\tilde{\mathfrak{K}}\|_{\tilde{B}^{\f{N}{2}-1,\f{N}{2}}}+\|\tilde{\mathfrak{J}}\|_{B^{\f{N}{2}-1}}+\|\tilde{\mathfrak{M}}\|_{B^{\f{N}{2}-1}}
+\|\tilde{\mathfrak{N}}\|_{B^{\f{N}{2}-1}}\Big)\\
&\quad+C\alpha_q
V'\Big(\|\tilde{\r}\|_{\tilde{B}^{\f{N}{2}-1,\f{N}{2}}}+\|\tilde{d}\|_{B^{\f{N}{2}-1}}+\|\tilde{E}^\top-\tilde{E}\|_{\tilde{B}^{\f{N}{2}-1,\f{N}{2}}}
+\|\tilde{\mathcal{E}}\|_{\tilde{B}^{\f{N}{2}-1,\f{N}{2}}}
+\|\tilde{\O}\|_{B^{\f{N}{2}-1}}\Big)\\&\quad-KV'h_q+\dl KV'+\dl \kappa.
\end{split}
\end{equation*}
Integrating over $[0,t]$ and having $\dl$ tend to $0$, we obtain
\begin{equation}\label{320}
\begin{split}
&g_q(t)+\kappa\int_0^tg_q(\tau)d\tau\\
&\le
g_q(0)+C\int_0^t\alpha_q(s)\Big(\|\tilde{\mathfrak{L}}\|_{\tilde{B}^{\f{N}{2}-1,\f{N}{2}}}+\|\tilde{\mathfrak{K}}\|_{\tilde{B}^{\f{N}{2}-1,\f{N}{2}}}
+\|\tilde{\mathfrak{J}}\|_{B^{\f{N}{2}-1}}+\|\tilde{\mathfrak{M}}\|_{B^{\f{N}{2}-1}}
+\|\tilde{\mathfrak{N}}\|_{B^{\f{N}{2}-1}}\\&\qquad+\|\tilde{\mathfrak{Q}}\|_{\tilde{B}^{\f{N}{2}-1,\f{N}{2}}}\Big)ds
+\int_0^tV'(s)\Big\{C\alpha_q(s)\Big(\|\tilde{\r}\|_{\tilde{B}^{\f{N}{2}-1,\f{N}{2}}}+\|\tilde{d}\|_{B^{\f{N}{2}-1}}
+\|\tilde{E}^\top-\tilde{E}\|_{\tilde{B}^{\f{N}{2}-1,\f{N}{2}}}\\&\qquad+\|\tilde{\mathcal{E}}\|_{\tilde{B}^{\f{N}{2}-1,\f{N}{2}}}+\|\tilde{\O}\|_{B^{\f{N}{2}-1}}\Big)-Kg_q(s)\Big\}ds.
\end{split}
\end{equation}
Thanks to \eqref{33}, we have
\begin{equation*}
\begin{split}
&C\alpha_q(s)\Big(\|\tilde{\r}\|_{\tilde{B}^{\f{N}{2}-1,\f{N}{2}}}+\|\tilde{d}\|_{B^{\f{N}{2}-1}}
+\|\tilde{E}^\top-\tilde{E}\|_{\tilde{B}^{\f{N}{2}-1,\f{N}{2}}}+\|\tilde{\mathcal{E}}\|_{\tilde{B}^{\f{N}{2}-1,\f{N}{2}}}+\|\tilde{\O}\|_{B^{\f{N}{2}-1}}
\Big)-Kg_q(s)\\&\quad\le
C\alpha_q(s)\|\tilde{\r}(s)\|_{\tilde{B}^{\f{N}{2}-1,\f{N}{2}}}-KC^{-1}2^{\phi^{\f{N}{2}-1,\f{N}{2}}(q)}\|\D_q\tilde{\r}\|_{L^2}
\\&\qquad+C\alpha_q(s)\|\tilde{E}^\top(s)-\tilde{E}(s)\|_{\tilde{B}^{\f{N}{2}-1,\f{N}{2}}}-KC^{-1}2^{\phi^{\f{N}{2}-1,\f{N}{2}}(q)}\|\D_q(\tilde{E}^\top-\tilde{E})\|_{L^2}
\\&\qquad+C\alpha_q(s)\|\tilde{\mathcal{E}}(s)\|_{\tilde{B}^{\f{N}{2}-1,\f{N}{2}}}-KC^{-1}2^{\phi^{\f{N}{2}-1,\f{N}{2}}(q)}\|\D_q\tilde{\mathcal{E}}\|_{L^2}
\\&\qquad+C\alpha_q(s)\|\tilde{d}(s)\|_{B^{\f{N}{2}-1}}-KC^{-1}2^{q(\f{N}{2}-1)}\|\D_q\tilde{d}\|_{L^2}\\
&\qquad+C\alpha_q(s)\|\tilde{\O}(s)\|_{B^{\f{N}{2}-1}}-KC^{-1}2^{q(\f{N}{2}-1)}\|\D_q\tilde{\O}\|_{L^2}.
\end{split}
\end{equation*}
If we choose $K\ge C^2$, we have
\begin{equation*}
\begin{split}
&\sum_{q\in\mathbb{Z}}\Big\{C\alpha_q(s)\Big(\|\tilde{\r}\|_{\tilde{B}^{\f{N}{2}-1,\f{N}{2}}}+\|\tilde{d}\|_{B^{\f{N}{2}-1}}
+\|\tilde{E}^\top-\tilde{E}\|_{\tilde{B}^{\f{N}{2}-1,\f{N}{2}}}+\|\tilde{\mathcal{E}}\|_{\tilde{B}^{\f{N}{2}-1,\f{N}{2}}}\\
&\qquad\qquad\qquad+\|\tilde{\O}\|_{B^{\f{N}{2}-1}}\Big)-Kg_q(s)\Big\}\le 0.
\end{split}
\end{equation*}

According to the last inequality, and thanks to \eqref{33} and
\eqref{320}, we conclude, after summation on $\mathbb{Z}$, that
\begin{equation}\label{321}
\begin{split}
&\|\tilde{\r}(t)\|_{\tilde{B}^{\f{N}{2}-1,\f{N}{2}}}+\|\tilde{E}(t)\|_{\tilde{B}^{\f{N}{2}-1,\f{N}{2}}}+\|\tilde{\O}(t)\|_{B^{\f{N}{2}-1}}
+\|\tilde{d}(t)\|_{B^{\f{N}{2}-1}}\\
&\quad+\kappa\left(\int_0^t\left(\|\tilde{\r}(\tau)\|_{\tilde{B}^{\f{N}{2}+1,\f{N}{2}}}+\|\tilde{E}(\tau)\|_{\tilde{B}^{\f{N}{2}+1,\f{N}{2}}}+\|\tilde{d}(\tau)\|_{\tilde{B}^{\f{N}{2}+1,\f{N}{2}-1}}
+\|\tilde{\O}(\tau)\|_{\tilde{B}^{\f{N}{2}+1,\f{N}{2}-1}}\right)d\tau\right)
\\&\le C\Big\{\|\r_0\|_{\tilde{B}^{\f{N}{2}-1,\f{N}{2}}}+\|E_0\|_{\tilde{B}^{\f{N}{2}-1,\f{N}{2}}}+\|\O_0\|_{B^{\f{N}{2}-1}}
+\|d_0\|_{B^{\f{N}{2}-1}}+\int_0^t\Big(\|\tilde{\mathfrak{L}}(s)\|_{\tilde{B}^{\f{N}{2}-1,\f{N}{2}}}
\\&\qquad+\|\tilde{\mathfrak{M}}(s)\|_{B^{\f{N}{2}-1}}+\|\tilde{\mathfrak{J}}(s)\|_{B^{\f{N}{2}-1}}+\|\tilde{\mathfrak{K}}\|_{\tilde{B}^{\f{N}{2}-1,\f{N}{2}}}
+\|\tilde{\mathfrak{N}}(s)\|_{B^{\f{N}{2}-1}}+\|\tilde{\mathfrak{Q}}\|_{\tilde{B}^{\f{N}{2}-1,\f{N}{2}}}\Big)ds\Big\}.
\end{split}
\end{equation}

{\bf Fourth Step: Smoothing Effect.}\quad
Once showing the damping effect for $\r$ and $E$, we can further
get the smoothing effect on $d$ and $\O$ for the system \eqref{31}
by considering the second and third equations with terms
$\Lambda\r$ and $\Lambda E$ being seen as given source terms.
Indeed, thanks to \eqref{321}, it suffices to state the proof for
high frequencies only. We therefore assume in this step
$q\ge q_0$.
Define
$$I_q=2^{q(\f{N}{2}-1)}\|\D_q\tilde{d}\|_{L^2}+2^{q(\f{N}{2}-1)}\|\D_q\tilde{\O}\|_{L^2}.$$
Then, from the energy estimates for the system
\begin{equation*}
\begin{cases}
\partial_t \D_q\tilde{d}+\D_q(\u\cdot\nabla \tilde{d})-\nu\D \D_q\tilde{d}=2\Lambda\D_q \tilde{\r}
+\D_q\tilde{\mathfrak{M}}-KV'(t)\D_q\tilde{d},\\
\partial_t\D_q\tilde{\O}+\D_q(\u\cdot\nabla\tilde{\O})-\mu\D\D_q\tilde{\O}=\mathcal{R}\D_q\tilde{E}+\D_q\tilde{\mathfrak{N}}
-KV'(t)\D_q\tilde{\O},
\end{cases}
\end{equation*}
we have
\begin{equation*}
\begin{split}
&\f{1}{2}\f{d}{dt}I_q^2+\kappa 2^{2q}I_q^2\\
&\le
I_q\Big(2^{q\f{N}{2}}\|\D_q\tilde{\r}\|_{L^2}+2^{q(\f{N}{2}-1)}\|\D_q\tilde{\mathfrak{M}}\|_{L^2}+2^{q(\f{N}{2}-1)}\|\D_q\tilde{\mathfrak{N}}\|_{L^2}
+2^{q\f{N}{2}}\|\D_q \tilde{E}\|_{L^2}\Big)\\&\qquad
+I_qV'(t)\Big(C\alpha_q(\|\tilde{d}\|_{B^{\f{N}{2}-1}}+\|\tilde{\O}\|_{B^{\f{N}{2}-1}})-KI_q\Big),
\end{split}
\end{equation*}
for a universal positive constant $\kappa$. Using
$J_q^2=I_q^2+\dl^2$, integrating over $[0,t]$ and then taking the
limit as $\dl\rightarrow 0$, we deduce
\begin{equation}\label{322}
\begin{split}
I_q(t)+\kappa 2^{2q}\int_0^tI_q(s)ds &\le I_q(0)+\int_0^t2^{q(\f{N}{2}-1)}(\|\D_q\tilde{\mathfrak{M}}(s)\|_{L^2}+\|\D_q\tilde{\mathfrak{N}}(s)\|_{L^2})ds\\
&\qquad+\int_0^t2^{q\f{N}{2}}(\|\D_q\tilde{\r}(s)\|_{L^2}+\|\D_q\tilde{E}(s)\|_{L^2})ds\\&\qquad
+C\int_0^tV'(s)\alpha_q(s)(\|\tilde{d}(s)\|_{B^{\f{N}{2}-1}}+\|\tilde{\O}(s)\|_{B^{\f{N}{2}-1}})ds.
\end{split}
\end{equation}
We therefore get
\begin{equation*}
\begin{split}
&\sum_{q\ge
q_0}2^{q(\f{N}{2}-1)}(\|\D_q\tilde{d}(t)\|_{L^2}+\|\D_q\tilde{\O}(t)\|_{L^2})+\kappa\int_0^t\sum_{q\ge
q_0}2^{q(\f{N}{2}+1)}(\|\D_q\tilde{d}(s)\|_{L^2}+\|\D_q\tilde{\O}(t)\|_{L^2})ds\\
&\quad\le\|d_0\|_{B^{\f{N}{2}-1}}+\|\O_0\|_{B^{\f{N}{2}}-1}+\int_0^t(\|\tilde{\mathfrak{M}}(s)\|_{B^{\f{N}{2}-1}}
+\|\tilde{\mathfrak{N}}(s)\|_{B^{\f{N}{2}-1}})ds\\&\qquad+\int_0^t\sum_{q\ge
q_0}2^{q\f{N}{2}}(\|\D_q\tilde{\r}(s)\|_{L^2}+\|\D_q\tilde{E}(s)\|_{L^2})ds+CV(t)\sup_{s\in
[0,t]}(\|\tilde{d}\|_{B^{\f{N}{2}-1}}+\|\tilde{\O}\|_{B^{\f{N}{2}-1}}).
\end{split}
\end{equation*}
Using \eqref{321}, we eventually conclude that
\begin{equation*}
\begin{split}
&\kappa\int_0^t\sum_{q\ge
q_0}2^{q(\f{N}{2}+1)}(\|\D_q\tilde{d}(s)\|_{L^2}+\|\D_q\tilde{\O}(s)\|_{L^2})ds\\&\quad\le
C(1+V(t))\Big\{\|\r_0\|_{\tilde{B}^{\f{N}{2}-1,\f{N}{2}}}+\|E_0\|_{\tilde{B}^{\f{N}{2}-1,\f{N}{2}}}+\|\O_0\|_{B^{\f{N}{2}-1}}+\|d_0\|_{B^{\f{N}{2}-1}}\\&\qquad\qquad
+\int_0^t\Big(\|\tilde{\mathfrak{L}}(s)\|_{\tilde{B}^{\f{N}{2}-1,\f{N}{2}}}+\|\tilde{\mathfrak{M}}(s)\|_{B^{\f{N}{2}-1}}
+\|\tilde{\mathfrak{N}}(s)\|_{B^{\f{N}{2}-1}}+\|\tilde{\mathfrak{K}}(s)\|_{\tilde{B}^{\f{N}{2}-1,\f{N}{2}}}\\
&\quad\qquad\qquad+\|\tilde{\mathfrak{Q}}(s)\|_{\tilde{B}^{\f{N}{2}-1,\f{N}{2}}}
+\|\tilde{\mathfrak{J}}(s)\|_{B^{\f{N}{2}-1}}\Big)ds\Big\}.
\end{split}
\end{equation*}
Combining the last inequality with \eqref{321}, we finish the
proof of Proposition \ref{p1}.
\end{proof}

In the rest of this section, we will sketch the proof of the
existence of a unique global solution to \eqref{31}. To this end,
we only need to use some properties of transport and heat
equations in nonhomogeneous Sobolev spaces $H^s$ with high
regularity order, see \cite{CH, RD}. Indeed, we set $(\r^0, d^0,
\O^0, E^0)=(\r_0, d_0, \O_0, E_0)$ and
\begin{equation}\label{323}
\begin{cases}
\partial_t \r^{n+1}+\u\cdot\nabla \r^{n+1}=-\Lambda d^n+\mathfrak{L},\\
\partial_t d^{n+1}-\nu\D d^{n+1}=\mathfrak{M}-\u\cdot\nabla d^n+2\Lambda \r^n,\\
\partial_t\O^{n+1}-\mu\D\O^{n+1}=\mathfrak{N}+\u\cdot\nabla\O^n+\Lambda ((E^n)^\top-E^n),\\
\partial_t ((E^{n+1})^\top-E^{n+1})+\u\cdot\nabla ((E^{n+1})^\top-E^{n+1})=-\Lambda\O^n+\mathfrak{Q},\\
\partial_t\mathcal{E}^{n+1}+\u\cdot\nabla\mathcal{E}^{n+1}=-2\Lambda
d^n+\mathfrak{K},
\end{cases}
\end{equation}
where

 $$(\r^{n+1}, d^{n+1}, \O^{n+1}, E^{n+1})|_{t=0}=(\r_0, d_0,
\O_0, E_0)$$ Let $T>0$ and $s$ (large enough) be fixed, and let
$K$ be a suitably large positive constant (depending on $s$, $T$
and $\u$). We set
$$\tilde{\r}^n=e^{-Kt}\r^n,\quad\tilde{d}^n=e^{-Kt}d^n,\quad\tilde{\O}^n=e^{-Kt}\O^n,\quad\tilde{E}^n=e^{-Kt}E^n.$$
The straightforward computations show that $\{(\tilde{\r}^n,
\tilde{d}^n, \tilde{\O}^n,\tilde{E}^n)\}_{n\in\mathbb{N}}$ is a
Cauchy sequence in
$$C([0,T]; H^s)\times(C([0,T]; H^{s-1}\times H^{s-1})\cap
L^1([0,T]; H^{s+1}\times H^{s+1}))^{2N}\times C([0,T];
H^s)^{N\times N}.$$ Denoting by $(\tilde{\r}, \tilde{d},
\tilde{\O},\tilde{E})$ the limit,  then it is easy to show that
$(e^{Kt}\tilde{\r}, e^{Kt}\tilde{d},
e^{Kt}\tilde{\O},e^{Kt}\tilde{E})$ solves \eqref{31}.

\bigskip

\section{Global Existence}

The goal of this section is to prove the global existence of
solutions to \eqref{e1} by building approximating solutions
$(\r^n,\u^n, E^n)$ using an iterative method. Those approximate
solutions are solutions of auxiliary systems of \eqref{31} to
which Proposition \ref{p1} applies.

We set the first term $(\r^0,\u^0, E^0)$ to $(0,0,0)$. We then
define $\{(\r^n,\u^n, E^n)\}_{n\in\mathbb{N}}$ by induction. In
fact, we choose $(\r^{n+1},\u^{n+1}, E^{n+1})$ as the solution of
the following  system:
\begin{equation}\label{41}
\begin{cases}
\partial_t \r^{n+1}+\u^n\cdot\nabla \r^{n+1}+\Lambda d^{n+1}=\mathfrak{L}^n,\\
\partial_t d^{n+1}+\u^n\cdot\nabla d^{n+1}-\nu\D d^{n+1}-2\Lambda \r^{n+1}=\mathfrak{M}^n,\\
\partial_t\O^{n+1}+\u^n\cdot\nabla\O^{n+1}-\mu\D\O^{n+1}-\Lambda (E^{n+1}-(E^{n+1})^\top)=\mathfrak{N}^n,\\
\partial_t ((E^{n+1})^\top-E^{n+1})+\u\cdot\nabla ((E^{n+1})^\top-E^{n+1})+\Lambda\O^{n+1}=\mathfrak{Q}^n,\\
\partial_t\mathcal{E}^{n+1}+\u\cdot\nabla\mathcal{E}^{n+1}+2\Lambda
d^{n+1}=\mathfrak{K}^n,\\
\u^{n+1}=-\Lambda^{-1}\nabla d^{n+1}-\Lambda^{-1}\textrm{curl}\O^{n+1}, \\
(\r^{n+1},d^{n+1}, \O^{n+1}, E^{n+1})|_{t=0}=(\r_n,
\Lambda^{-1}\Dv \u_n, \Lambda^{-1}\mathrm{curl}\u_n, E_n),
\end{cases}
\end{equation}
where
$$\r_n=\sum_{|q|\le n}\D_q\r_0,\quad \u_n=\sum_{|q|\le
n}\D_q\u_0,\quad E_n=\sum_{|q|\le n}\D_qE_0, $$
$$\mathfrak{L}^n=-\r^n\Dv\u^n,$$
\begin{equation*}
\begin{split}
\mathfrak{M}^n=\u^n\cdot\nabla
d^n-\Lambda^{-1}\Dv\Big(&\u^n\cdot\nabla\u^n+K(\r^n)\nabla
\r^n+\f{\r^n}{1+\r^n}\mathcal{A}\u^n\\
&-E^n_{jk}\partial_{x_j}E_{ik}^n+\Dv(\r^n
E^n)\Big),
\end{split}
\end{equation*}
$$\mathfrak{N}^n=\u^n\cdot\nabla\O^n-\Lambda^{-1}\mathrm{curl}\left(\u^n\cdot\nabla\u^n+K(\r^n)\nabla
\r^n+\f{\r^n}{1+\r^n}\mathcal{A}\u^n-E^n_{jk}\partial_{x_j}E_{ik}^n\right),$$
$$\mathfrak{Q}^n=(\nabla\u^n E^n)^\top-\nabla\u^n E^n,$$
and
\begin{equation*}
\begin{split}
\mathfrak{K}^n&=\u^n\cdot\nabla\mathcal{E}^{n}-\Lambda^{-1}\partial_{x_i}\Lambda^{-1}\partial_{x_j}(\u\cdot\nabla(E^n_{ij}+E^n_{ji}))\\&\quad+\Lambda^{-1}\partial_{x_i}\Lambda^{-1}\partial_{x_j}((\nabla\u^n
E^n)_{ij}+(\nabla\u^n E^n)_{ji}).
\end{split}
\end{equation*}

As in \eqref{6402}, due to the divergence property of the
deformation gradient, we can rewritten the second equation in
\eqref{41} as
\begin{equation}\label{x41}
\partial_t d^{n+1}+\u\cdot\nabla d^{n+1}-\nu\D d^{n+1}-2\Lambda\mathcal{E}^{n+1}=\mathfrak{J}^n,
\end{equation}
where $\mathfrak{J}^n$ is given by
\begin{equation*}
\begin{split}
\mathfrak{J}^n&=\u^n\cdot\nabla
d^n-\Lambda^{-1}\Dv\Big(\u^n\cdot\nabla\u^n+K(\r^n)\nabla
\r^n+\f{\r^n}{1+\r^n}\mathcal{A}\u^n-E^n_{jk}\partial_{x_j}E_{ik}^n-\Dv(\r^n
E^n)\Big).
\end{split}
\end{equation*}

The argument in the previous section guarantees that the system
\eqref{41} is solvable and Remark \ref{X1} tells us that, in view
of \eqref{x41}, the solution to \eqref{41} satisfies the estimates
in Proposition \ref{p1}.

\subsection{Uniform Estimates in the Critical Regularity Case}
In this subsection, we establish uniform estimates in
$\mathfrak{B}^{\f{N}{2}}$ for $(\r^n, \u^n, E^n)$. Denote
$$\gamma=\|\r_0\|_{\tilde{B}^{\f{N}{2}-1,\f{N}{2}}}+\|\u_0\|_{B^{\f{N}{2}-1}}+\|E_0\|_{\tilde{B}^{\f{N}{2}-1,\f{N}{2}}}.$$
We are going to show that there exists a positive constant
$\Gamma$ such that, if $\gamma$ is small enough, the following
estimate holds for all $n\in\mathbb{N}$:
\begin{equation}
\|(\r^n, \u^n, E^n)\|_{\mathfrak{B}^{\f{N}{2}}}\le \Gamma \gamma.
\tag{$\mathfrak{P}_n$}
\end{equation}
We will prove ($\mathfrak{P}_n$) by the mathematical induction.
Suppose that ($\mathfrak{P}_n$) is satisfied and let us prove that
($\mathfrak{P}_{n+1}$) also holds.

According to Proposition \ref{p1} and the definition of
$(\r_n, \u_n, E_n)$, the following inequality holds
\begin{equation}\label{40}
\begin{split}
\|(\r^{n+1}, \u^{n+1}, E^{n+1})\|_{\mathfrak{B}^{\f{N}{2}}}&\le
Ce^{CV^n}\Big(\|\r_0\|_{\tilde{B}^{\f{N}{2}-1,\f{N}{2}}}+\|\u_0\|_{B^{\f{N}{2}-1}}+\|E_0\|_{\tilde{B}^{\f{N}{2}-1,\f{N}{2}}}\\
&\quad+\|\mathfrak{L}^n\|_{L^1(\tilde{B}^{\f{N}{2}-1,\f{N}{2}})}+\|\mathfrak{M}^n\|_{L^1(B^{\f{N}{2}-1})}+\|\mathfrak{N}^n\|_{L^1(B^{\f{N}{2}-1})}
\\&\qquad+\|\mathfrak{Q}^n\|_{L^1(\tilde{B}^{\f{N}{2}-1,\f{N}{2}})}+\|\mathfrak{K}^n\|_{L^1(\tilde{B}^{\f{N}{2}-1,\f{N}{2}})}
+\|\mathfrak{J}^n\|_{L^1(B^{\f{N}{2}-1})}\Big)
\end{split}
\end{equation}
where
$$V^n=\int_0^\infty\|\u^n(s)\|_{B^{\f{N}{2}+1}}ds.$$
Hence, it remains to obtain the estimates for $\mathfrak{L}^n$,
$\mathfrak{M}^n$, $\mathfrak{N}^n$, $\mathfrak{K}^n$,
$\mathfrak{J}^n$ and $\mathfrak{Q}^n$ by using ($\mathfrak{P}_n$).

{\bf Estimate of $\mathfrak{L}^n$:} The estimate of
$\mathfrak{L}^n$ is straightforward; thanks to Proposition
\ref{p2}, we have
\begin{equation}\label{44}
\begin{split}
\|\mathfrak{L}^n\|_{L^1(\tilde{B}^{\f{N}{2}-1,\f{N}{2}})}&\le
C\|\r^n\|_{L^\infty(\tilde{B}^{\f{N}{2}-1,\f{N}{2}})}\|\Dv\u^n\|_{L^1(B^{\f{N}{2}})}\\
&\le C\Gamma^2\gamma^2.
\end{split}
\end{equation}

{\bf Estimates of $\mathfrak{M}^n$, $\mathfrak{N}^n$, and
$\mathfrak{J}^n$:}
To estimate $\mathfrak{M}^n$, $\mathfrak{N}^n$, and
$\mathfrak{J}^n$, we assume that $\gamma$ satisfies
$$\gamma\le \f{1}{2\Gamma\mathfrak{C}},$$
where $\mathfrak{C}$ is the continuity modulus of
$B^{\f{N}{2}-1,\f{N}{2}}\hookrightarrow L^\infty$. If
($\mathfrak{P}_n$) holds, then
\begin{equation*}
\|\r^n\|_{L^\infty(\R^+\times\R^N)}\le\mathfrak{C}\|\r^n\|_{B^{\f{N}{2}-1,\f{N}{2}}}\le
\f{1}{2},
\end{equation*}
and
$$\|E^n\|_{L^\infty(\R^+\times\R^N)}\le\mathfrak{C}\|E^n\|_{B^{\f{N}{2}-1,\f{N}{2}}}\le
\f{1}{2}.$$
 We will estimate $\mathfrak{M}^n$ term by term.
First we have, according to Proposition \ref{p2} and
Lemma \ref{l11},
\begin{equation}\label{43}
\begin{split}
\left\|\f{\r^n}{1+\r^n}\nabla^2\u^n\right\|_{L^1(B^{\f{N}{2}-1})}&\le
C\|\nabla^2\u^n\|_{L^1(B^{\f{N}{2}-1})}\left\|\f{\r^n}{1+\r^n}\right\|_{L^\infty(B^{\f{N}{2}})}\\
&\le
C\|\u^n\|_{L^1(B^{\f{N}{2}+1})}\|\r^n\|_{L^\infty(B^{\f{N}{2}})}
\le C\Gamma^2\gamma^2
\end{split}
\end{equation}
Since $K(0)=0$ and Lemma \ref{l11}, one has
\begin{equation*}
\begin{split}
\|K(\r^n)\nabla\r^n\|_{L^1(B^{\f{N}{2}-1})}\le
C\|K(\r^n)\|_{L^2(B^{\f{N}{2}})}\|\nabla\r^n\|_{L^2(B^{\f{N}{2}-1})}
\le C\|\r^n\|_{L^2(B^{\f{N}{2}})}^2.
\end{split}
\end{equation*}
On the other hand, we have, by H\"{o}lder's inequality,
\begin{equation*}
\begin{split}
\|\r^n\|^2_{L^2(B^{\f{N}{2}})}&=\int_0^\infty\left(\sum_{q\in\mathbb{Z}}\left(2^{q\phi^{\f{N}{2}-1,\f{N}{2}}(q)}\|\D_q\r^n(t)\|_{L^2}\right)^{\f12}
\left(2^{q\phi^{\f{N}{2}+1,\f{N}{2}}(q)}\|\D_q\r^n(t)\|_{L^2}\right)^{\f12}\right)^2dt\\
&\le\int_0^\infty\|\r^n(t)\|_{\tilde{B}^{\f{N}{2}-1,\f{N}{2}}}\|\r^n(t)\|_{\tilde{B}^{\f{N}{2}+1,\f{N}{2}}}dt\\
&\le\|\r^n\|_{L^\infty(\tilde{B}^{\f{N}{2}-1,\f{N}{2}})}\|\r^n\|_{L^1(\tilde{B}^{\f{N}{2}+1,\f{N}{2}})}
\le C\Gamma^2\gamma^2.
\end{split}
\end{equation*}
Thus, the above two inequalities imply that
\begin{equation}\label{46}
\begin{split}
\|K(\r^n)\nabla\r^n\|_{L^1(B^{\f{N}{2}-1})}\le C\Gamma^2\gamma^2.
\end{split}
\end{equation}
From Proposition \ref{p2}, we obtain the following estimates:
\begin{equation}\label{47}
\begin{split}
&\|\u^n\cdot\nabla
d^n\|_{L^1(B^{\f{N}{2}-1})}+\|\u^n\cdot\nabla\u^n\|_{L^1(B^{\f{N}{2}-1})}+\|\u^n\cdot\nabla\O^n\|_{L^1(B^{\f{N}{2}-1})}\\
&\quad\le
C\int_0^\infty\|\u^n(t)\|_{B^{\f{N}{2}-1}}\left(\|d^n(t)\|_{B^{\f{N}{2}+1}}+\|\O^n(t)\|_{B^{\f{N}{2}+1}}\right)dt\\
&\quad\le
C\|\u^n\|_{L^\infty(B^{\f{N}{2}-1})}\left(\|d^n\|_{L^1(B^{\f{N}{2}+1})}+\|\O^n\|_{L^1(B^{\f{N}{2}+1})}\right)\\
&\quad\le C\Gamma^2\gamma^2;
\end{split}
\end{equation}
\begin{equation}\label{48}
\begin{split}
\|E^n_{jk}\partial_{x_j}E_{ik}^n\|_{L^1(B^{\f{N}{2}-1})}&\le
C\int_0^\infty\|E^n(t)\|_{\tilde{B}^{\f{N}{2}-1,\f{N}{2}}}\|E^n(t)\|_{\tilde{B}^{\f{N}{2}+1,\f{N}{2}}}dt\\
&\le C\|E^n\|_{L^\infty(\tilde{B}^{\f{N}{2}-1,\f{N}{2}})}\|E^n\|_{L^1(\tilde{B}^{\f{N}{2}+1,\f{N}{2}})}\\
&\le C\Gamma^2\gamma^2;
\end{split}
\end{equation}
and
\begin{equation}\label{49}
\begin{split}
\|\Dv(\r^n E^n)\|_{L^1(B^{\f{N}{2}-1})}&\le
C\Big(\int_0^\infty\|\r^n(t)\|_{\tilde{B}^{\f{N}{2}-1,\f{N}{2}}}\|E^n(t)\|_{\tilde{B}^{\f{N}{2}+1,\f{N}{2}}}dt\\&\quad+
\int_0^\infty\|E^n(t)\|_{\tilde{B}^{\f{N}{2}-1,\f{N}{2}}}\|\r^n(t)\|_{\tilde{B}^{\f{N}{2}+1,\f{N}{2}}}dt\Big)\\
&\le C\Big(\|\r^n\|_{L^\infty(\tilde{B}^{\f{N}{2}-1,\f{N}{2}})}\|E^n\|_{L^1(\tilde{B}^{\f{N}{2}+1,\f{N}{2}})}\\
&\quad+\|E^n\|_{L^\infty(\tilde{B}^{\f{N}{2}-1,\f{N}{2}})}\|\r^n\|_{L^1(\tilde{B}^{\f{N}{2}+1,\f{N}{2}})}\Big)\\
&\le C\Gamma^2\gamma^2.
\end{split}
\end{equation}
Summarizing \eqref{43}-\eqref{49}, we finally get
\begin{equation}\label{410}
\begin{split}
\|\mathfrak{M}^n\|_{L^1(B^{\f{N}{2}-1})}+\|\mathfrak{N}^n\|_{L^1(B^{\f{N}{2}-1})}+\|\mathfrak{J}^n\|_{L^1(B^{\f{N}{2}-1})}\le
C\Gamma^2\gamma^2.
\end{split}
\end{equation}

{\bf Estimates of $\mathfrak{Q}^n$ and $\mathfrak{K}^n$:}  It is
easy to obtain, using Proposition \ref{p2},
\begin{equation}\label{411}
\begin{split}
\|\nabla\u^n E^n\|_{L^1(\tilde{B}^{\f{N}{2}-1,\f{N}{2}})}&\le
C\int_0^\infty\|\u^n(t)\|_{B^{\f{N}{2}+1}}\|E^n(t)\|_{\tilde{B}^{\f{N}{2}-1,\f{N}{2}}}dt\\
&\le
C\|\u^n\|_{L^1(B^{\f{N}{2}+1})}\|E^n\|_{L^\infty(\tilde{B}^{\f{N}{2}-1,\f{N}{2}})}\\
&\le C\Gamma^2\gamma^2.
\end{split}
\end{equation}
Hence
$$\|\mathfrak{Q}^n\|_{\tilde{B}^{\f{N}{2}-1,\f{N}{2}}}+\|\mathfrak{K}^n\|_{\tilde{B}^{\f{N}{2}-1,\f{N}{2}}}\le
C\Gamma\gamma.$$

From \eqref{44}, \eqref{410} and \eqref{411}, we finally have
\begin{equation}\label{412}
\begin{split}
\|(\r^{n+1}, \u^{n+1}, E^{n+1})\|_{\mathfrak{B}^{\f{N}{2}}}\le
Ce^{C\Gamma\gamma}(\Gamma^2\gamma^2+\gamma).
\end{split}
\end{equation}
Now we choose $\Gamma=4C$ and choose $\gamma$ such that
\begin{equation}
\Gamma^2\gamma\le 1,\quad e^{C\Gamma\gamma}\le
2\quad\textrm{and}\quad \Gamma\gamma\le \f{1}{2\mathfrak{C}},
\tag{$\mathfrak{H}$}
\end{equation}
then ($\mathfrak{P}_n$) holds for all $n\in \mathbb{N}$.

\medskip\medskip

\subsection{Existence of a solution}
In this subsection, we show that, up to a subsequence, the
sequence $(\r^n, \u^n, E^n)_{n\in\mathbb{N}}$ converges in
$\mathcal{D}'(\R^+\times\R^N)$ to a solution $(\r,\u, E)$ of
\eqref{e1}. We will use some compactness arguments. The starting
point is to show that the  first-order time derivative of $(\r^n,
\u^n, E^n)$ is uniformly bounded in appropriate spaces. This
enables us to apply the Ascoli-Arzela theorem and get the
existence of a limit $(\r,\u,E)$ for a subsequence. Then, the
uniform bounds of the previous subsection provide us with some
additional regularity and convergence properties so that we may
pass to the limit in the system \eqref{41}.

To begin with, we have to prove uniform bounds in the suitable
functional spaces and the convergence of $\{(\r^n, \u^n,
E^n)\}_{n\in\mathbb{N}}$. The uniform bounds is summarized in the
following lemma.

\begin{Lemma}\label{ub}
$\{(\r^n, \u^n, E^n)\}_{n\in\mathbb{N}}$ is uniformly bounded in
$$C^{\f{1}{2}}_{loc}\left(\R^+;
B^{\f{N}{2}-1}\right)\times\left(C^{\f{1}{4}}_{loc}\left(\R^+;B^{\f{N}{2}-\f{3}{2}}\right)\right)^N\times\left(C^{\f{1}{2}}_{loc}\left(\R^+;
B^{\f{N}{2}-1}\right)\right)^{N\times N}$$ (and also in
$C^{\f{1}{2}}_{loc}\left(B^{\f{N}{2}-1}\times\left(B^{\f{N}{2}-2}\right)^N\times
\left(B^{\f{N}{2}-1}\right)^{N\times N}\right)$ if $N\ge 3$).
\end{Lemma}
\begin{proof}We will finish the proof via five steps.

{\bf Step 1: Uniform bound of $\partial_t\r^n$ in
$L^2(B^{\f{N}{2}-1})$.} In fact, notice that
$$\partial_t\r^{n+1}=-\r^n\Dv\u^n-\u^n\cdot\nabla \r^{n+1}-\Lambda
d^n.$$ According to the estimates in the previous subsection and
the interpolation result in Proposition \ref{p3}, we see that
$\{\u^n\}_{n\in\mathbb{N}}$ is uniformly bounded in
$L^2(B^{\f{N}{2}})$, and $\{(\r^n, E^n)\}_{n\in\mathbb{N}}$ is
uniformly bounded in $(L^\infty(B^{\f{N}{2}}))^{N^2+1}$. Thus,
$-\r^n\Dv\u^n-\u^n\cdot\nabla \r^{n+1}-\Lambda d^n$ is uniformly
bounded in $L^2(B^{\f{N}{2}-1})$,  which implies that
$\partial_t\r^n$ is uniformly bounded in $L^2(B^{\f{N}{2}-1})$, and furthermore
$\{\r^n\}_{n\in\mathbb{N}}$ is uniformly bounded in
$C^{\f{1}{2}}_{loc}\left(\R^+; B^{\f{N}{2}-1}\right)$.

{\bf Step 2: Uniform bound of $\partial_tE^n$ in
$L^2(B^{\f{N}{2}-1})$.} In fact, notice that
$$\partial_t((E^{n+1})^\top-E^{n+1})=-\Lambda\O^n+\nabla\u^n E^n-\u^n\cdot\nabla ((E^{n+1})^\top-E^{n+1}).$$ According to the estimates in the previous subsection and
the interpolation result in Proposition \ref{p3}, we conclude that
$\{\u^n\}_{n\in\mathbb{N}}$ is uniformly bounded in
$L^2(B^{\f{N}{2}})$, and $\{(\r^n, E^n)\}_{n\in\mathbb{N}}$ is
uniformly bounded in $(L^\infty(B^{\f{N}{2}}))^{N^2+1}$. Thus,
$-\Lambda\O^n+\nabla\u^n E^n-\u^n\cdot\nabla
((E^{n+1})^\top-E^{n+1})$ is uniformly bounded in
$L^2(B^{\f{N}{2}-1})$. Therefore, $\partial_tE^n$ is
uniformly bounded in $L^2(B^{\f{N}{2}-1})$, and furthermore,
$\{(E^n)^\top-E^n\}_{n\in\mathbb{N}}$ is uniformly bounded in
$\left(C^{\f{1}{2}}_{loc}\left(\R^+;
B^{\f{N}{2}-1}\right)\right)^{N\times N}$.
Similarly, we can show that $\{\mathcal{E}^n\}_{n\in\mathbb{N}}$
is uniformly bounded in $C^{\f{1}{2}}_{loc}\left(\R^+;
B^{\f{N}{2}-1}\right)$. Hence, $\{E^n\}_{n\in\mathbb{N}}$ is
uniformly bounded in $C^{\f{1}{2}}_{loc}\left(\R^+;
B^{\f{N}{2}-1}\right)$.

{\bf Step 3: Uniform bound of $\partial_td^n$ in
$L^{\f{4}{3}}(B^{\f{N}{2}-\f{3}{2}})+L^4(B^{\f{N}{2}-\f{3}{2}})$.}
To this end, we recall that
\begin{equation*}
\begin{split}
\partial_t d^{n+1}&=-\u^n\cdot\nabla d^{n+1}+\nu\D d^{n+1}+2\Lambda
\r^{n+1}+\u^n\cdot\nabla d^n\\&\quad
-\Lambda^{-1}\Dv\Big(\u^n\cdot\nabla\u^n+K(\r^n)\nabla
\r^n+\f{\r^n}{1+\r^n}\mathcal{A}\u^n-E^n_{jk}\partial_{x_j}E_{ik}^n+\Dv(\r^n
E^n)\Big).
\end{split}
\end{equation*}
The estimates in the previous subsection and the interpolation
result in Proposition \ref{p3} yield that
$\{\u^n\}_{n\in\mathbb{N}}$ is uniformly bounded in
$L^\infty(B^{\f{N}{2}-1})\cap L^{\f{4}{3}}(B^{\f{N+1}{2}})$. This
uniform bound, combining together with the uniform bound of
$\{\r^n\}_{n\in\mathbb{N}}$ in $L^\infty(B^{\f{N}{2}})$, gives an
uniform bound of
$$-\u^n\cdot\nabla d^{n+1}+\nu\D d^{n+1}+\u^n\cdot\nabla
d^n-\Lambda^{-1}\Dv\left(\u^n\cdot\nabla\u^n+\f{\r^n}{1+\r^n}\mathcal{A}\u^n\right)$$
in $L^{\f{4}{3}}(B^{\f{N-3}{2}})$. Using the
uniform bound of $\{\r^n\}_{n\in\mathbb{N}}$ in
$L^\infty(B^{\f{N}{2}})\cap L^2(B^{\f{N}{2}})$ obtained
from the uniform bounds of $\{\r^n\}_{n\in\mathbb{N}}$ in
$L^\infty(\tilde{B}^{\f{N}{2}-1,\f{N}{2}})\cap
L^1(\tilde{B}^{\f{N}{2}+1,\f{N}{2}})$ and the interpolation result
in Proposition \ref{p3}, we deduce that
$\{\r^n\}_{n\in\mathbb{N}}$ is uniformly bounded in
$L^4(B^{\f{N-1}{2}})$, and hence
$\{\Lambda\r^n\}_{n\in\mathbb{N}}$ is uniformly bounded in
$L^4(B^{\f{N-3}{2}})$. Finally, both
$E^n_{jk}\partial_{x_j}E_{ik}^n$ and $\Dv(\r^n E^n)$ are also
uniformly bounded in $L^4(B^{\f{N-3}{2}})$, since
$\{\r^n\}_{n\in\mathbb{N}}$ and $\{E^n\}_{n\in\mathbb{N}}$ are
uniformly bounded in $L^\infty(B^{\f{N}{2}})\cap
L^4(B^{\f{N-1}{2}})$ and $\Dv(\r E)$ can be rewritten as the sum
of $\r\Dv E$ and $\nabla\r E$. Therefore, $\{\partial_t
d^n\}_{n\in\mathbb{N}}$ is uniformly bounded in
$L^{\f{4}{3}}(B^{\f{N}{2}-\f{3}{2}})+L^4(B^{\f{N}{2}-\f{3}{2}})$.

{\bf Step 4: Uniform bound of $\partial_t\O^n$ in
$L^{\f{4}{3}}(B^{\f{N}{2}-\f{3}{2}})+L^4(B^{\f{N-3}{2}})$.} We
recall that
\begin{equation*}
\begin{split}
\partial_t\O^{n+1}&=-\u^n\cdot\nabla\O^{n+1}+\ov{\mu}\D\O^{n+1}+\Lambda((E^{n+1})^\top-E^{n+1})+\u^n\cdot\nabla\O^n\\&\quad-\Lambda^{-1}\mathrm{curl}\left(\u^n\cdot\nabla\u^n+K(\r^n)\nabla
\r^n+\f{\r^n}{1+\r^n}\mathcal{A}\u^n-E^n_{jk}\partial_{x_j}E_{ik}^n\right).
\end{split}
\end{equation*}
Similar to the argument in Step 3, we conclude that
$$-\u^n\cdot\nabla d^{n+1}+\nu\D d^{n+1}+\u^n\cdot\nabla
d^n-\Lambda^{-1}\textrm{curl}\left(\u^n\cdot\nabla\u^n+\f{\r^n}{1+\r^n}\mathcal{A}\u^n\right)$$
is uniformly bounded in $L^{\f{4}{3}}(B^{\f{N-3}{2}})$, using the
uniform bound of $\{\r^n\}_{n\in\mathbb{N}}$ in
$L^\infty(B^{\f{N}{2}})\cap L^2(B^{\f{N}{2}})$. Also,  similarly to Step 3,  $\{\Lambda
E^{n+1}\}_{n\in\mathbb{N}}$ and
$\{E^n_{jk}\partial_{x_j}E_{ik}^n\}_{n\in\mathbb{N}}$ are
uniformly bounded in $L^4(B^{\f{N-3}{2}})$. Hence, $\{\partial_t
\O^n\}_{n\in\mathbb{N}}$ is uniformly bounded in
$L^{\f{4}{3}}(B^{\f{N}{2}-\f{3}{2}})$
$+L^4(B^{\f{N}{2}-\f{3}{2}})$.

{\bf Step 5: Uniform bound as $N\ge 3$.} Indeed, in this case, the
only difference is the uniform bound on $\u^n$. Actually,
from the uniform bounds on $\{\u^n\}_{n\in\mathbb{N}}$ in
$L^\infty(B^{\f{N}{2}-1})\cap L^2(B^{\f{N}{2}})$, we deduce that
$\{\u^n\cdot\nabla\u^n\}_{n\in\mathbb{N}}$ is uniformly bounded in
$L^2(B^{\f{N}{2}-2})$. Then, following the same argument in Step 3
and Step 4, we can deduce that
$\{\partial_td^n\}_{n\in\mathbb{N}}$ and
$\{\partial_t\O^n\}_{n\in\mathbb{N}}$ are uniformly bounded in
$L^2(B^{\f{N}{2}-2})+L^\infty(B^{\f{N}{2}-2})$, because
$\{\Lambda\r^n\}_{n\in\mathbb{N}}$, $\{\Lambda
E^{n+1}\}_{n\in\mathbb{N}}$ and
$\{E^n_{jk}\partial_{x_j}E_{ik}^n\}_{n\in\mathbb{N}}$ are
uniformly bounded in $L^\infty(B^{\f{N}{2}-2})$. This further
implies that $\{\partial_t\u^n\}_{n\in\mathbb{N}}$ is uniformly
bounded in $L^2(B^{\f{N}{2}-2})+L^\infty(B^{\f{N}{2}-2})$, and
hence is uniformly bounded in $C^{\f{1}{2}}_{loc}(\R^+;
B^{\f{N}{2}-2})$.
\end{proof}

We can now turn to the proof of the existence in  Theorem \ref{mt}, using
a compactness argument. To this end, we denote
$\{\chi_p\}_{p\in\mathbb{N}}$ be a sequence of $C_0^\infty(\R^N)$
cut-off functions supported in the ball $B(0, p+1)$ in $\R^N$ and
equal to $1$ in a neighborhood of the ball $B(0,p)$.
For any $p\in\mathbb{N}$, Lemma \ref{ub} tells us that
$\{(\chi_p\r^n, \chi_p\u^n, \chi_p E^n)\}_{n\in\mathbb{N}}$ is
uniformly equicontinuous in $C(\R^+; B^{\f{N}{2}-1}\times
(B^{\f{N-3}{2}})^N\times (B^{\f{N}{2}-1})^{N\times N})$. Notice
that the operator $f\mapsto \chi_p f$ is compact from
$B^{\f{N}{2}-1}\cap B^{\f{N}{2}}$ into $\dot{H}^{\f{N}{2}-1}$, and
from $B^{\f{N}{2}-1}\cap B^{\f{N-3}{2}}$ into
$\dot{H}^{\f{N-3}{2}}$. This can be proved easily by noticing that
$f\mapsto \chi_p f$ is compact from $\dot{H}^s\cap\dot{H}^{s'}$
into $\dot{H}^s$ for $s<s'$ and $B^s\hookrightarrow\dot{H}^s$. We
now apply the Ascoli-Arzela theorem to the sequence
$\{(\chi_p\r^n, \chi_p\u^n, \chi_p E^n)\}_{n\in\mathbb{N}}$ on the
time interval $[0,p]$. We then use Cantor's diagonal process. This
finally provides us with a distribution $(\r,\u,E)$ belonging to
$C(\R^+;
\dot{H}^{\f{N}{2}-1}\times(\dot{H}^{\f{N-3}{2}})^N\times(\dot{H}^{\f{N}{2}-1})^{N\times
N})$ and a subsequence (which we still denote by $\{(\r^n, \u^n,
E^n)\}_{n\in\mathbb{N}}$), such that, for all $p\in\mathbb{N}$, we
have
\begin{equation}\label{cs}
\begin{split}
(\chi_p\r^n, \chi_p\u^n, \chi_pE^n)\rightarrow (\chi_p\r,
\chi_p\u, \chi_pE)
\end{split}
\end{equation}
in
$C([0,p];\dot{H}^{\f{N}{2}-1}\times(\dot{H}^{\f{N-3}{2}})^N\times(\dot{H}^{\f{N}{2}-1})^{N\times
N})$. In particular, this implies that $(\r^n,\u^n, E^n)$ tends to
$(\r,\u, E)$ in $\mathcal{D}'(\R^+\times\R^N)$ as $n\rightarrow
\infty$. Furthermore, according to the uniform bounds in the
previous subsection, we deduce that $(\r, \u, E)$ belongs to
\begin{equation*}
\begin{split}
&L^\infty\left(\R^+;
\tilde{B}^{\f{N}{2}-1,\f{N}{2}}\times(B^{\f{N}{2}-1})^N\times(\tilde{B}^{\f{N}{2}-1,\f{N}{2}})^{N\times
N}\right)\\&\quad\cap L^1\left(\R^+;
\tilde{B}^{\f{N}{2}+1,\f{N}{2}}\times(B^{\f{N}{2}+1})^N\times(\tilde{B}^{\f{N}{2}+1,\f{N}{2}})^{N\times
N}\right),
\end{split}
\end{equation*}
and belongs to $C^{\f{1}{2}}(\R^+; B^{\f{N}{2}-1})\times
(C^{\f{1}{4}}(\R^+; B^{\f{N-3}{2}}))^N\times(C^{\f{1}{2}}(\R^+;
B^{\f{N}{2}-1}))^{N\times N}$ (and also belongs to
$C^{\f{1}{2}}(\R^+; B^{\f{N}{2}-1}\times (B^{\f{N}{2}-2})^N\times
(B^{\f{N}{2}-1})^{N\times N})$ if $N\ge 3$). And, obviously, we
have the bounds provided by $(\mathfrak{P_n})$ for this solution.
More important, we can claim that the whole sequence is a Cauchy
sequence in $\mathfrak{B}^{\f{N}{2}}$ as the initial data is such
small that $2C\Gamma\gamma<\f{1}{2}$. Indeed, if we denote
$$U^n=\|(\r^n-\r^{n-1}, \u^n-\u^{n-1},
E^n-E^{n-1})\|_{\mathfrak{B}^{\f{N}{2}}},$$ then applying
Proposition \ref{p2} and following the estimates in the previous
subsection, we can deduce that $U^n$ satisfies
$$U^n(t)\le a_n+\f{1}{2}U^{n-1}(t),\quad\textrm{for all}\quad t\in
\R^+,$$ where $a_n\ge 0$, obtained from the initial data,
satisfies $\sum_{n\in\mathbb{N}}a_n<\infty$. This inequality
easily ensure that the sequence $\{(\r^n, \u^n,
E^n)\}_{n\in\mathbb{N}}$ is a Cauchy sequence in
$\mathfrak{B}^{\f{N}{2}}$.

Next, we need to prove that $(\r,\u, E)$ obtained above solves
\eqref{e1e}. To this end, we first observe that, according to
\eqref{6401},
\begin{equation}\label{413}
\begin{cases}
\partial_t \r^{n+1}+\u^n\cdot\nabla \r^{n+1}+\Dv \u^{n+1}=-\r^n\Dv\u^n,\\
\partial_t\u^{n+1}+\u^n\cdot\nabla\u^{n+1}-\mathcal{A}\u^{n+1}+\nabla\r^{n+1}-\Dv E^{n+1}\\\quad=K(\r^n)\nabla\r^n+\f{\r^n}{1+\r^n}\mathcal{A}\u^n+\Dv((E^n)(E^n)^\top),\\
\partial_t E^{n+1}+\u^n\cdot\nabla E^{n+1}+\nabla\u^n=\nabla\u^n E^{n+1}.\\
\end{cases}
\end{equation}
Hence, the only problem now is to pass to the limit in
$\mathcal{D}'(\R^+\times\R^N)$ for each term in \eqref{413},
especially for those nonlinear terms. Let $\theta\in
C_0^\infty(\R^+\times\R^N)$ and $p\in\mathbb{N}$ be such that
$\textrm{supp }\theta\subset[0,p]\times B(0,p)$. For the
convergence of $\f{\r^n}{1+\r^n}\mathcal{A}\u^n$, we write
\begin{equation*}
\begin{split}
\theta\left(\f{\r^n}{1+\r^n}\mathcal{A}\u^n-\f{\r}{1+\r}\mathcal{A}\u\right)&=\theta\f{\r^n}{1+\r^n}\chi_p\mathcal{A}(\chi_p\u^n-\chi_p\u)
\\&\quad+\theta\left(\f{\chi_p\r^n}{1+\chi_p\r^n}-\f{\chi_p\r}{1+\chi_p\r}\right)\chi_p\mathcal{A}\u.
\end{split}
\end{equation*}
Since $\theta \f{\r^n}{1+\r^n}$ is uniformly bounded in
$L^\infty(B^{\f{N}{2}})\subset L^\infty(\dot{H}^{\f{N}{2}})$ and
$\chi_p\u^n$ tends to $\chi_p\u$ in
$C([0,p];\dot{H}^{\f{N-3}{2}})$, the first term in the above
identity tends to $0$ in $C([0,p];\dot{H}^{\f{N-3}{2}})$, while
$\f{\chi_p\r^n}{1+\chi_p\r^n}$ tends to $\f{\chi_p\r}{1+\chi_p\r}$
in $C([0,p]; \dot{H}^{\f{N}{2}-1})$ by Lemma \ref{l11}, which
implies that the third term also tends to $0$ in $C([0,p];
\dot{H}^{\f{N-3}{2}})$. The convergence of other nonlinear terms
can be treated similarly.

Finally, we now prove the continuity of $\r$ and $E$ in time.

\begin{Lemma} Let $(\r,\u, E)$ be a solution to \eqref{e1e}. Then $\r$ and $E$ are continuous in
$\tilde{B}^{\f{N}{2}-1,\f{N}{2}}$, and $\u$ belongs to $C(\R^+;
B^{\f{N}{2}-1})$.
\end{Lemma}
\begin{proof} To prove this, we follow the argument in \cite{RD1}.
Indeed, the continuity of $\u$ is straightforward, because $\u$
satisfies
$$\partial_t\u=-\u\cdot\nabla\u+\mathcal{A}\u-\nabla\r-\left(\f{\r}{1+\r}\right)\mathcal{A}\u-K(\r)\nabla\r+E_{jk}\nabla_{x_j}E_{ik},$$
and the right-hand side belongs to
$L^1(B^{\f{N}{2}-1})+L^2(B^{\f{N}{2}-1})$ due to the facts that
$\r\in L^2(B^{\f{N}{2}-1})$ and $E\in L^2(B^{\f{N}{2}-1})$.

To prove  $\r,E\in C(\R^+; B^{\f{N}{2}-1})$, we notice that,
$\r_0, E_0\in B^{\f{N}{2}-1}$, $\r, E\in L^\infty(\R^+;
B^{\f{N}{2}-1})$ and $\partial_t\r, \partial_tE\in L^2(\R^+;
B^{\f{N}{2}-1})$. Thus, it remains to prove the continuity in time
of $\r, E$ in $B^{\f{N}{2}}$.
To this end, we apply the operator $\D_q$ to the first equation of
\eqref{e1e} to yield
\begin{equation}\label{414}
\begin{split}
\partial_t\D_q\r=-\D_q(\u\cdot\nabla\r)-\Lambda\D_q
d-\D_q(\r\Dv\u).
\end{split}
\end{equation}
Obviously, for fixed $q$, the right-hand side belongs to
$L^1(\R^+; L^2)$. Hence, each $\D_q\r$ is continuous in time with
values in $L^2$ (thus in $B^{\f{N}{2}}$).
Now, applying an energy method to \eqref{414}, thanks to Lemma
\ref{cl}, we obtain
\begin{equation*}
\begin{split}
\f{1}{2}\f{d}{dt}\|\D_q\r\|_{L^2}^2\le
C\|\D_q\r\|_{L^2}\Big(\alpha_q
2^{-q\f{N}{2}}\|\r\|_{B^{\f{N}{2}}}\|\u\|_{B^{\f{N}{2}+1}}+\|\Lambda\D_qd\|_{L^2}+\|\D_q(\r\Dv\u)\|_{L^2}\Big).
\end{split}
\end{equation*}
Integrating the above inequality with respect to time on the
interval $[t_1, t_2]$, we obtain
\begin{equation*}
\begin{split}
2^{q\f{N}{2}}\|\D_q\r(t_2)\|_{L^2}&\le
2^{q\f{N}{2}}\|\D_q\r(t_1)\|_{L^2}+C\int_{t_1}^{t_2}\Big(\alpha_q(\tau)\|\r(\tau)\|_{B^{\f{N}{2}}}\|\u(\tau)\|_{B^{\f{N}{2}+1}}\\&\qquad+2^{q(\f{N}{2}+1)}\|\D_qd(\tau)\|_{L^2}
+2^{q\f{N}{2}}\|\D_q(\r\Dv\u)(\tau)\|_{L^2}\Big)d\tau
\end{split}
\end{equation*}
Since $\r\in L^\infty(B^{\f{N}{2}})$, $\u\in L^1(B^{\f{N}{2}+1})$,
and $\r\Dv\u\in L^1(B^{\f{N}{2}})$, we eventually obtain
\begin{equation*}
\begin{split}
\|\r(t_2)\|_{B^{\f{N}{2}}}\lesssim\|\r(t_1)\|_{B^{\f{N}{2}}}+(1+\|\r\|_{L^\infty(B^{\f{N}{2}})})\int_{t_1}^{t_2}\|\u(\tau)\|_{B^{\f{N}{2}+1}}d\tau+\int_{t_1}^{t_2}\|\r\Dv\u(\tau)\|_{B^{\f{N}{2}}}d\tau,
\end{split}
\end{equation*}
which implies that $\r$ belongs to $C(\R^+; B^{\f{N}{2}})$.

Similarly, we can prove that $E$ also belongs to $C(\R^+;
B^{\f{N}{2}})$.

This finishes our proof.
\end{proof}

\bigskip\bigskip

\section{Appendix}

For the completeness of the presentation, we give in this appendix
the proof of two fundamental lemmas concerning the divergence and
curl of the deformation gradient in the viscoelasticity system
\eqref{e1e}.

The first one we are going to state now is the following lemma
(cf. Proposition 3.1 in \cite{LZ}).

\begin{Lemma}\label{div}
Assume that $\Dv(\r_0\F_0^\top)=0$ and $(\r, \u, \F)$ is the
solution of the system \eqref{e1e}. Then the following identity
\begin{equation}\label{Dv}
\Dv(\r \F^\top)=0
\end{equation}
holds for all time $t> 0$.
\end{Lemma}
\begin{proof}
First, we transpose the third equation in \eqref{e1e} and apply
the divergence operator to the resulting equation to yield
\begin{equation}\label{61}
\partial_t(\partial_{x_j}\F_{ji})+\u\cdot\nabla(\partial_{x_j}\F_{ji})=\left(\f{\partial^2\u_j}{\partial{x_k}\partial{x_j}}\right)\F_{ki}.
\end{equation}
Multiply the first equation in \eqref{e1e} by
$\partial_{x_j}\F_{ji}$, multiply \eqref{61} by $\r$, and summing
them together, we obtain
\begin{equation}\label{62}
\partial_t(\r\partial_{x_j}\F_{ji})+\u\cdot\nabla(\r\partial_{x_j}\F_{ji})=\r\left(\f{\partial^2\u_j}{\partial{x_k}\partial{x_j}}\right)\F_{ki}
-\r\partial_{x_k}\u_k\partial_{x_j}\F_{ji}.
\end{equation}

On the other hand, we differentiate the first equation in
\eqref{e1e} with respect to $x_j$ to yield
\begin{equation}\label{63}
\partial_t(\partial_{x_j}\r)+\u\cdot\nabla(\partial_{x_j}\r)+\partial_{x_j}\u_k\partial_{x_k}\r+\f{\partial^2\u_k}{\partial
x_k\partial x_j}\r+\partial_{x_j}\r\partial_{x_k}\u_k=0.
\end{equation}
Multiplying \eqref{63} by $\F_{ji}$, multiplying the third
equation in \eqref{e1e} by $\partial_{x_j}\r$, and summing them
together, we obtain
\begin{equation}\label{64}
\partial_t(\partial_{x_j}\r\F_{ji})+\u\cdot\nabla(\partial_{x_j}\r\F_{ji})=-\r\left(\f{\partial^2\u_k}{\partial{x_k}\partial{x_j}}\right)\F_{ji}
-\partial_{x_j}\r\partial_{x_k}\u_k\F_{ji}.
\end{equation}

Adding \eqref{62} and \eqref{64} together yields
\begin{equation}\label{65}
\partial_t(\Dv(\r\F^\top))+\Dv(\u\otimes\Dv(\r\F^\top))=0.
\end{equation}

If $(\r, \u, \F)$ is sufficiently smooth, we multiply \eqref{65}
by $\Dv(\r\F^\top)$, we get
$$\partial_t\left(\left|\Dv(\r\F^\top)\right|^2\right)+\Dv\left(\u\left|\Dv(\r\F^\top)\right|^2\right)=-\f{1}{2}\Dv\u\left|\Dv(\r\F^\top)\right|^2.$$
Integrating the above identity with respect to $x$ over $\R^N$, we
obtain
\begin{equation*}
\begin{split}
\f{d}{dt}\left\|\Dv(\r\F^\top)\right\|_{L^2}^2&=-\f{1}{2}\int_{\R^N}\Dv\u\left|\Dv(\r\F^\top)\right|^2
dx\\
&\le
\f{1}{2}\|\nabla\u\|_{L^\infty}\left\|\Dv(\r\F^\top)\right\|_{L^2}^2,
\end{split}
\end{equation*}
which, by Gronwall's inequality, implies that, for all $t\ge 0$
$$\left\|\Dv(\r\F^\top)(t)\right\|_{L^2}^2\le\left\|\Dv(\r_0\F_0^\top)\right\|_{L^2}^2e^{\f{1}{2}\int_0^t\|\nabla\u(\tau)\|_{L^\infty}d\tau}.$$
Hence, if $\Dv(\r_0\F_0^\top)=0$, the above inequality will gives
$\|\Dv(\r\F^\top)\|_{L^2}=0$ for all $t>0$, which implies that
$\Dv(\r\F^\top)=0$ for all $t>0$.

This finishes the proof.
\end{proof}

Another hidden, but important, property of the viscoelastic fluids
system \eqref{e1e} is concerned with the curl of the deformation
gradient (for the incompressible case, see \cite{LLZH2, LLZH}).
Actually, the following lemma says that the curl of the
deformation gradient is of higher order.

\begin{Lemma}\label{curl}
Assume that \eqref{e1e3} is satisfied and $(\u, \F)$ is the
solution of the system \eqref{e1e}. Then the following identity
\begin{equation}\label{curl1}
\F_{lk}\nabla_l \F_{ij}=\F_{lj}\nabla_l \F_{ik}
\end{equation}
holds for all time $t> 0$ if it initially satisfies \eqref{curl1}.
\end{Lemma}

\begin{proof}
First, we establish the evolution equation for the equality
$\F_{lk}\nabla_l \F_{ij}-\F_{lj}\nabla_l \F_{ik}$. Indeed, by the
equation \eqref{e1e3}, we can get
\begin{equation*}
\partial_t\nabla_l\F_{ij}+\u\cdot\nabla\nabla_l
\F_{ij}+\nabla_l\u\cdot\nabla\F_{ij}=\nabla_m\u_i\nabla_l\F_{mj}+\nabla_l\nabla_m\u_i
\F_{mj}.
\end{equation*}
Thus,
\begin{equation}\label{curl2}
\F_{lk}(\partial_t\nabla_l\F_{ij}+\u\cdot\nabla\nabla_l
\F_{ij})+\F_{lk}\nabla_l\u\cdot\nabla\F_{ij}=\F_{lk}\nabla_m\u_i\nabla_l\F_{mj}+\F_{lk}\nabla_l\nabla_m\u_i
\F_{mj}.
\end{equation}
Also, from \eqref{e1e3}, we obtain
\begin{equation}\label{curl3}
\nabla_l\F_{ij}(\partial_t\F_{lk}+\u\cdot\nabla\F_{lk})=\nabla_l\F_{ij}\nabla_m\u_l\F_{mk}.
\end{equation}

Now, adding \eqref{curl2} and \eqref{curl3}, we deduce that
\begin{equation}\label{curl4}
\begin{split}
\partial_t(\F_{lk}\nabla_l \F_{ij})+\u\cdot\nabla(\F_{lk}\nabla_l
\F_{ij})&=-\F_{lk}\nabla_l\u\cdot\nabla\F_{ij}+\F_{lk}\nabla_m\u_i\nabla_l\F_{mj}\\&\quad+\F_{lk}\nabla_l\nabla_m\u_i
\F_{mj}+\nabla_l\F_{ij}\nabla_m\u_l\F_{mk}\\
&=\F_{lk}\nabla_m\u_i\nabla_l\F_{mj}+\F_{lk}\nabla_l\nabla_m\u_i
\F_{mj}.
\end{split}
\end{equation}
Here, we used the identity which is derived by interchanging the
roles of indices $l$ and $m$:
$$\F_{lk}\nabla_l\u\cdot\nabla\F_{ij}=\F_{lk}\nabla_l\u_m\nabla_m\F_{ij}=\nabla_l\F_{ij}\nabla_m\u_l\F_{mk}.$$
Similarly, one has
\begin{equation}\label{curl5}
\begin{split}
\partial_t(\F_{lj}\nabla_l \F_{ik})+\u\cdot\nabla(\F_{lj}\nabla_l
\F_{ik})=\F_{lj}\nabla_m\u_i\nabla_l\F_{mk}+\F_{lj}\nabla_l\nabla_m\u_i
\F_{mk}.
\end{split}
\end{equation}
Subtracting \eqref{curl5} from \eqref{curl4} yields
\begin{equation}\label{curl6}
\begin{split}
&\partial_t(\F_{lk}\nabla_l \F_{ij}-\F_{lj}\nabla_l
\F_{ik})+\u\cdot\nabla(\F_{lk}\nabla_l \F_{ij}-\F_{lj}\nabla_l
\F_{ik})\\&\quad=\nabla_m\u_i(\F_{lk}\nabla_l\F_{mj}-\F_{lj}\nabla_l\F_{mk})+\nabla_l\nabla_m\u_i
(\F_{mj}\F_{lk}-\F_{mk}\F_{lj}).
\end{split}
\end{equation}
Due to the fact
$$\nabla_l\nabla_m\u_i=\nabla_m\nabla_l\u_i$$ in the sense of distributions, we
have, again by interchanging the roles of indices $l$ and $m$,
\begin{equation*}
\begin{split}
\nabla_l\nabla_m\u_i
(\F_{mj}\F_{lk}-\F_{mk}\F_{lj})&=\nabla_l\nabla_m\u_i
\F_{mj}\F_{lk}-\nabla_l\nabla_m\u_i
\F_{mk}\F_{lj}\\
&=\nabla_l\nabla_m\u_i \F_{mj}\F_{lk}-\nabla_m\nabla_l\u_i
\F_{lk}\F_{mj}\\
&=(\nabla_l\nabla_m\u_i-\nabla_m\nabla_l\u_i) \F_{lk}\F_{mj}=0.
\end{split}
\end{equation*}
From this identity, equation \eqref{curl6} can be simplified as
\begin{equation}\label{curl7}
\begin{split}
&\partial_t(\F_{lk}\nabla_l \F_{ij}-\F_{lj}\nabla_l
\F_{ik})+\u\cdot\nabla(\F_{lk}\nabla_l \F_{ij}-\F_{lj}\nabla_l
\F_{ik})\\&\quad=\nabla_m\u_i(\F_{lk}\nabla_l\F_{mj}-\F_{lj}\nabla_l\F_{mk}).
\end{split}
\end{equation}
Multiplying \eqref{curl7} by $\F_{lk}\nabla_l
\F_{ij}-\F_{lj}\nabla_l \F_{ik}$, we get
\begin{equation}\label{curl8}
\begin{split}
&\partial_t|\F_{lk}\nabla_l \F_{ij}-\F_{lj}\nabla_l
\F_{ik}|^2+\u\cdot\nabla|\F_{lk}\nabla_l \F_{ij}-\F_{lj}\nabla_l
\F_{ik}|^2\\&\quad=2(\F_{lk}\nabla_l \F_{ij}-\F_{lj}\nabla_l
\F_{ik})\nabla_m\u_i(\F_{lk}\nabla_l\F_{mj}-\F_{lj}\nabla_l\F_{mk})\\
&\quad\le 2\|\nabla\u\|_{L^\infty(\R^3)}\mathcal{M}^2,
\end{split}
\end{equation}
where $\mathcal{M}$ is defined as
$$\mathcal{M}=\max_{i,j,k}\{|\F_{lk}\nabla_l
\F_{ij}-\F_{lj}\nabla_l \F_{ik}|^2\}.$$ Hence, \eqref{curl8}
implies
\begin{equation}\label{curl9}
\partial_t\mathcal{M}+\u\cdot\nabla\mathcal{M}\le
2\|\nabla\u\|_{L^\infty(\R^3)}\mathcal{M}.
\end{equation}

On the other hand, the characteristics of $\partial_t
f+\u\cdot\nabla f=0$ is given by
$$\f{d}{ds}X(s)=\u(s,X(s)),\quad X(t)=x.$$
Hence, \eqref{curl8} can be rewritten as
\begin{equation}\label{curl10}
\f{\partial U}{\partial t}\le B(t,y)U,\quad
U(0,y)=\mathcal{M}_0(y),
\end{equation}
where
$$U(t,y)=\mathcal{M}(t, X(t,x)),\quad
B(t,y)=2\|\nabla\u\|_{L^\infty(\R^3)}(t, X(t,y)).$$ The
differential inequality \eqref{curl10} implies that
$$U(t,y)\le U(0)\exp\left(\int_0^t B(s,y)ds\right).$$ Hence,
$$\mathcal{M}(t,x)\le \mathcal{M}(0)\exp\left(\int_0^t
2\|\nabla\u\|_{L^\infty(\R^3)}(s)ds\right).$$ Hence, if
$\mathcal{M}(0)=0$, then $\mathcal{M}(t)=0$ for all $t>0$, and
 the proof of the lemma is complete.
\end{proof}

Using $\F=I+E$, \eqref{curl1} means
\begin{equation}\label{11111b}
\nabla_{k}E_{ij}+E_{lk}\nabla_{l}E_{ij}=\nabla_{j}E_{ik}+E_{lj}\nabla_{l}E_{ik}.
\end{equation}
According to \eqref{11111b}, it is natural to assume that the
initial condition of $E$ in the viscoelastic fluids system
\eqref{e1} should satisfy the compatibility condition
\begin{equation}\label{11113}
\nabla_{k}E(0)_{ij}+E(0)_{lk}\nabla_{l}E(0)_{ij}=\nabla_{j}E(0)_{ik}+E(0)_{lj}\nabla_{l}E(0)_{ik}.
\end{equation}

\bigskip

\section*{Acknowledgments}

Xianpeng Hu's research was supported in part by the National
Science Foundation grant DMS-0604362. Dehua Wang's research was
supported in part by the National Science Foundation under grants
DMS-0604362 and DMS-0906160, and by the Office of Naval Research
under Grant N00014-07-1-0668.


\bigskip

\begin{thebibliography}{999}

\bibitem{CH} Chemin, J. Y.,
\emph{Perfect incompressible fluids}. Translated from the 1995
French original by Isabelle Gallagher and Dragos Iftimie. Oxford
Lecture Series in Mathematics and its Applications, 14. The
Clarendon Press, Oxford University Press, New York, 1998.

\bibitem{RD} Danchin, R., \emph{Fourier Analysis Methods for PDE'S}, lecture notes, 2005.

\bibitem{RD3} Danchin, R.,
\emph{On the uniqueness in critical spaces for compressible
Navier-Stokes equations}. NoDEA Nonlinear Differential Equations
Appl. 12 (2005), 111--128.

\bibitem{RD1} Danchin, R.,
\emph{Global existence in critical spaces for flows of
compressible viscous and heat-conductive gases}. Arch. Ration.
Mech. Anal. 160 (2001), 1--39.

\bibitem{RD2} Danchin, R.,
\emph{Global existence in critical spaces for compressible
Navier-Stokes equations}. Invent. Math. 141 (2000), 579--614.

\bibitem{JP} Peetre, J.,
\emph{New thoughts on Besov spaces}, Duke University Mathematics
Series, No. 1. Mathematics Department, Duke University, Durham,
N.C., 1976.

\bibitem{CM} Chemin, J.; Masmoudi, N.,
\emph{About lifespan of regular solutions of equations related to
viscoelastic fluids.} SIAM J. Math. Anal. 33 (2001), 84--112.

\bibitem{CZ} Chen, Y.; Zhang, P.,
\emph{The global existence of small solutions to the
incompressible viscoelastic fluid system in 2 and 3 space
dimensions.} Comm. Partial Differential Equations 31 (2006),
1793--1810.

\bibitem{CD} Dafermos, C. M.,
\emph{Hyperbolic conservation laws in continuum physics.} Second
edition. Grundlehren der Mathematischen Wissenschaften, 325.
Springer-Verlag, Berlin, 2005.

\bibitem{HW} Hu, X., Wang, D.,\emph{Local Strong Solution to the Compressible Viscoelastic Flow with Large
Data}. To appear in J. Differential Equations.

\bibitem{Gurtin}
 Gurtin, M. E., {\em An introduction to Continuum Mechanics.} Mathematics in Science and Engineering,158. Academic Press, New York-London, 1981.

\bibitem{KP} Kessenich, P., \emph{Global Existence with Small Initial Data for Three-Dimensional Incompressible Isotropic Viscoelastic Materials,}  preprint.

\bibitem{LLZH2} Lei, Z.; Liu, C.; Zhou, Y.,
\emph{Global existence for a 2D incompressible viscoelastic model
with small strain.} Commun. Math. Sci. 5 (2007), 595--616.

\bibitem{LLZ} Lei, Z.; Liu, C.; Zhou, Y.,
\emph{Global solutions for incompressible viscoelastic fluids.}
Arch. Ration. Mech. Anal. 188 (2008), 371--398.

\bibitem{LZ} Lei, Z.; Zhou, Y.,
\emph{Global existence of classical solutions for the
two-dimensional Oldroyd model via the incompressible limit.} SIAM
J. Math. Anal. 37 (2005), 797--814.

\bibitem{LLZH} Lin, F; Liu, C.; Zhang, P.,
\emph{On hydrodynamics of viscoelastic fluids.} Comm. Pure Appl.
Math. 58 (2005), 1437--1471.

\bibitem{LZP} Lin, F; Zhang, P.,
\emph{On the initial-boundary value problem of the incompressible
viscoelastic fluid system.} Comm. Pure Appl. Math. 61 (2008),
539--558.

\bibitem{LM} Lions, P. L.; Masmoudi, N.,
\emph{Global solutions for some Oldroyd models of non-Newtonian
flows.} Chinese Ann. Math. Ser. B 21 (2000), 131--146.

\bibitem{LW} Liu, C.; Walkington, N. J.,
\emph{An Eulerian description of fluids containing visco-elastic
particles.} Arch. Ration. Mech. Anal. 159 (2001), 229--252.

\bibitem{QZ} Qian, J., Zhang, Z., \emph{Global well-posedness for
the compressible viscoelastic fluids near equilibrium}. Preprint.

\bibitem{RHN}
Renardy, M.; Hrusa, W. J.;  Nohel, J. A., {\em  Mathematical
Problems in Viscoelasticity.}
 Longman Scientific and Technical; copublished in the US with John Wiley, New York, 1987.

\bibitem{ST} Sideris, T. C.; Thomases, B.,
\emph{Global existence for three-dimensional incompressible
isotropic elastodynamics via the incompressible limit.} Comm. Pure
Appl. Math. 58 (2005), 750--788.

\end{thebibliography}
\end{document}